\documentclass[11pt,a4paper,reqno]{article}
\usepackage[utf8]{inputenc}
\usepackage[T1]{fontenc}
\usepackage[english]{babel}
\usepackage{graphicx,xcolor}
\usepackage{amsfonts,amssymb,amsbsy,latexsym,a4wide,xspace}
\usepackage{amsmath,delarray,enumerate,calc,amsthm}
\usepackage{bbm}
\usepackage{txfonts}
\usepackage{paralist}
\usepackage{authblk}
\usepackage{mathabx}  
\usepackage{tikz-cd}
\usetikzlibrary{matrix}
\usepackage{url}
\usepackage{hyperref}
\usepackage{comment}

\newcommand{\cA}{{\mathcal A}}

\newcommand{\cB}{{\mathcal B}}

\newcommand{\cF}{{\mathcal F}}
\newcommand{\cG}{{\mathcal G}}
\newcommand{\cH}{{\mathcal H}}

\newcommand{\cP}{{\mathcal P}}

\newcommand{\wq}{\tilde{\tau}}

\newcommand{\np}{s}

\newcommand{\cM}{{\mathcal M}}

\newcommand{\cU}{{\mathcal U}}

\newcommand{\N}{{\mathbbm N}}

\newcommand{\Z}{{\mathbbm Z}}

\newcommand{\1}{{\mathbbm 1}}

\newcommand{\id}{\operatorname{id}}
\newcommand{\inn}{\operatorname{int}}
\newcommand{\lcm}{\operatorname{lcm}}
\newcommand{\card}[1]{|#1|}
\newcommand{\ddelta}{\boldsymbol{\delta}}
\newcommand{\Spec}{\operatorname{Spec}}
\renewcommand{\mod}{\text{ mod }}
\newcommand{\sat}{{sat}}

\newcommand{\Sh}{(Sh)}

\newcommand{\wcH}{{\widetilde\cH}}

\newtheorem {lemma}{Lemma}[section]
\newtheorem{theorem}[lemma]{Theorem}
\newtheorem{proposition}[lemma]{Proposition}
\newtheorem {corollary}[lemma]{Corollary}
\newtheorem {definition}[lemma]{Definition}

\newtheorem {bemerkung}[lemma]{Remark}
\newtheorem{beispiel}[lemma]{Example}

\newenvironment{remark} {\begin{bemerkung} \normalfont }{\end{bemerkung}}
\newenvironment{example} {\begin{beispiel} \normalfont }{\end{beispiel}}

\newcommand{\gen}[1]{\langle #1\rangle}
\newcommand{\Aut}{\operatorname{Aut}}
\newcommand{\Tor}{\operatorname{Tor}}

\makeatletter
\newcommand{\subjclass}[2][2020]{%
  \let\@oldtitle\@title%
  \gdef\@title{\@oldtitle\footnotetext{#1 \emph{MSC classification:} #2}}%
}
\newcommand{\keywords}[1]{%
  \let\@@oldtitle\@title%
  \gdef\@title{\@@oldtitle\footnotetext{\emph{Keywords:} #1.}}%
}
\makeatother
\begin{document}
\title{Automorphisms of $\cB$-free and other Toeplitz shifts}
\author{Aurelia Dymek$^1$ \quad Stanis\l{}aw Kasjan$^1$ \quad Gerhard Keller$^2$}
\affil{\small
$^1$\ Faculty of Mathematics and Computer Science, Nicolaus Copernicus University, Toru\'n, Poland\\
$^2$\ Department of Mathematics, Friedrich-Alexander University, Erlangen, Germany
}
\date{Version of \today}
\subjclass{37A44, 37B05, 37B10}
\keywords{$\cB$-free dynamics, Toeplitz dynamical system, sets of multiples, automorphisms group, trivial centralizer}

\maketitle
\begin{abstract}
We present sufficient conditions for the triviality of the automorphism group of regular Toeplitz subshifts and give a broad class of examples from the class of $\cB$-free subshifts satisfying them, extending \cite{Bartnicka2017}.
On the other hand we provide an example of a $\cB$-free Toeplitz subshift whose
automorphism group has elements of arbitrarily large finite order, answering
Question 11 in \cite{MR3821717}.
\end{abstract}
\section{Introduction}

\paragraph{Toeplitz subshifts}
Let $\eta\in\{0,1\}^\Z$ be a non-periodic Toeplitz sequence with period structure
$p_1\mid p_2\mid p_3 \dots\to\infty$. That means for each $k\in\Z$ there exists $n\ge1$ such that $\eta_{|k+p_n\Z}$ is constant, but $\eta$ is not periodic (the latter excludes trivial cases). Denote the orbit closure of $\eta$ under the left shift $\sigma$ on $\{0,1\}^\Z$ by $X_\eta$. Each such subshift is called a \emph{Toeplitz shift}.
It follows from the work of Williams \cite{MR756807} that such systems are in fact almost 1-1 extensions of their \emph{maximal equicontinuous factor (MEF)}, in this case of an associated odometer $(G,T)$, where $G$ is the compact topological group $\varprojlim \Z/p_n\Z$ built from the period structure
$(p_n)_{n\ge 1}$, and
$T$ is the translation by $(1, 1,\ldots)$.
 Recall that odometers are minimal, equicontinuous and zero-dimensional dynamical systems, and the conjunction of these three properties characterizes them among all topological dynamical systems, see \cite{Downarowicz2005}.

\paragraph{The centralizer}
For any subshift $(Y,\sigma)$, i.e., a shift invariant closed subset $Y\subseteq\{0,1\}^\Z$, the \emph{automorphism group} (or \emph{centralizer}) is the group of all homeomorphisms $U\colon Y\to Y$ which commute with $\sigma$. Its elements are sliding block codes \cite{MR0259881}. Therefore the automorphism group is countable. Since all powers of the shift are elements of the centralizer, it contains a copy of $\Z$ as a normal subgroup. We say that the automorphism group is \emph{trivial} if it consists solely of powers of the shift.

Centralizers are studied for various classes of systems.
Bu\l{}atek and Kwiatkowski \cite{Bulatek1990} do it for Toeplitz subshifts with separated holes {\Sh} using elements of the associated odometer $(G,T)$. Moreover, in \cite{MR1199322} they deliver examples of Toeplitz subshifts with positive topological entropy and trivial automorphism group. More recently, Cyr and Kra study automorphism groups of subshifts of subquadratic and linear growth in \cite{MR3430839, MR3324942}. In \cite{MR3430839}, they prove that the cosets of powers of the shift in the centralizer of any topologically transitive subshift of subquadratic growth form a periodic group. In \cite{MR3324942}, they show that any minimal subshift with non-superlinear complexity has a virtually $\Z$ automorphism group, answering the question asked in \cite{MR3394119}. In \cite{Donoso2015} the same result is shown independently with different methods by Donoso, Durand, Maass and Petite. On the other hand, the centralizer can be quite a complicated group. In \cite{MR3570202} Salo gives an example of a Toeplitz subshift with not finitely generated automorphism group. We provide another example with this property, see~\eqref{eq:ex-B_1^N} and Corollary~\ref{coro:large-centralizer}.

\paragraph{This paper's contributions to general Toeplitz shifts}
Let $\eta$ be a Toeplitz sequence with period structure $(p_n)_{n\ge 1}$ and recall from
\cite{Bulatek1990} that a position $k\in\Z$ is called a \emph{hole at level $N$} if $\eta_{|k+p_N\Z}$ is not constant. We refine this concept and call such a hole
\emph{essential}, if the residue class $k+p_N\Z$ contains holes of each level $n\ge N$, see Definition~\ref{def:essential-holes}. The minimal period $\wq_N$ of the set of essential holes at level $N$ divides $p_N$ and, under a (seemingly strong) additional assumption, Theorem~\ref{theo:basic-case} provides restrictions on the size of the centralizer in terms of the quotients $p_N/\wq_N$. A direct application of this result to a variant of the Garcia - Hedlund sequence is discussed in Example~\ref{ex:GarciaHedlung}.
After that, using a mixture of topological and arithmetic arguments, we show that the additional assumption is satisfied more often than one may expect -- the main tool is Theorem~\ref{theo:pre-B-free}. Along this way we exploit suitable topological variants of the separated holes condition \eqref{Sh} from \cite{Bulatek1990} in Proposition~\ref{prop:local-invariance-new}, and verify these conditions under arithmetic assumptions tailor-made for the $\cB$-free case in Proposition~\ref{prop:local-invariance-new-2}. We note that our techniques generalize the setting from \cite{Bulatek1990}, but are kind of transverse to the setting from \cite{MR1199322}, see Remark~\ref{remark:condition(*)}.

\paragraph{The centralizer of $\cB$-free Toeplitz shifts}
For any set $\cB\subseteq\{2,3,\ldots\}$
let
\begin{equation*}
\cM_\cB:=\bigcup_{b\in\cB}b\Z
\end{equation*}
be the set of \emph{multiples of $\cB$} and
\begin{equation*}
\cF_\cB:=\Z\setminus\cM_\cB
\end{equation*}
the set of \emph{$\cB$-free numbers}. One can easily modify a set $\cB$ to have the same set of multiples and to be \emph{primitive}, i.e., $b\nmid b'$ for different $b,b'\in\cB$. So, we will tacitly assume that $\cB$ is primitive.
The investigation of sets of multiples and $\cB$-free numbers has a quite long history, see \cite{Hall1996}. Recently, the subshifts associated with $\cB$-free numbers are under intensive study, see e.g. \cite{BKKL2015,KKL2016} and references therein. Namely, let $\eta\in\{0,1\}^\Z$ be the characteristic function of the $\cB$-free numbers.
Theorem B in \cite{KKL2016} shows that minimality of $(X_\eta,\sigma)$ is equivalent to $\eta$ being Toeplitz, whenever $\cB$ is taut. The tautness assumption was removed in \cite[Thm.~3.7]{DKKP2022}. So, in the $\cB$-free context, $(X_\eta,\sigma)$ is minimal if and only if $\eta$ is a Toeplitz sequence.

In the case of $\cB$-free Toeplitz subshifts the previously cited results for low complexity systems are not useful, because even very simple $\cB$-free Toeplitz systems may have superpolynomial complexity, see Section~\ref{sec:complexity} below, where this is shown for $\cB=\cB_1:=\{2^nc_n:n\in\N\}$ with pairwise coprime odd $c_n$. Although the separated holes condition {\Sh} is satisfied for this simple example, we were not able to use the results from \cite{Bulatek1990} to determine the centralizer. However, its triviality was shown with more direct methods in \cite{Bartnicka2017}.
On the other hand, there are many $\cB$-free Toeplitz systems which do not satisfy {\Sh} anyway, see the example discussed at the end of this introduction, so that there is need for techniques relying neither on low complexity nor on {\Sh}.

We mention briefly that
Mentzen \cite{Mentzen2017} proves the triviality of the automorphism group for any Erd\H{o}s $\cB$-free subshift, i.e., when $\cB$ is infinite, pairwise coprime and $\sum_{b\in\cB}\frac{1}b<\infty$. This is extended to taut $\cB$ containing an infinite pairwise coprime subset in \cite{Keller2016, KLRS2021}. This class of $\cB$-free subshifts is kind of opposite to the $\cB$-free Toeplitz shifts.

\paragraph{This paper's contributions to $\cB$-free Toeplitz shifts}
In Section~\ref{sec:Bfree}, the result from Section~\ref{sec:abstract_Toeplitz} are applied to $\cB$-free examples.
This is possible, because Theorem~\ref{theo:ess-holes-arithmetic} provides an arithmetic characterization of the sets of essential holes in terms of sets of multiples derived explicitly from the set $\cB$.
Then we can use the general results to
produce examples of minimal $\cB$-free systems, including not only the case $\cB=\cB_1$ treated in \cite{Bartnicka2017} but also many systems violating the separated holes condition {\Sh}, which have trivial centralizers,
see Subsection~\ref{subsec:trivial-centralizer}. (The reader will notice that a very broad class of examples can be treated along the same lines.)
The fact that the general results fail to guarantee triviality of the centralizer for some examples, to which even the basic Theorem~\ref{theo:basic-case} applies, is not a shortcoming of our approach. This is illustrated in subsection~\ref{subsec:non-trivial}, where we consider
simple variants of $\cB=\cB_1$, still satisfying condition {\Sh}, but having non-trivial centralizers -- just as big as Theorem~\ref{theo:basic-case} allows them to be.
This provides a negative answer to Question 11 in \cite{MR3821717}. Indeed, a slight generalization of this construction provides examples for which the centralizer contains elements of arbitrarily large finite order, see
Remark~\ref{rem:finite-order}. It should be noticed that our examples have superpolynomial complexity, see Proposition \ref{proposition_complexity}, so that the complexity based results from the literature discussed above do not apply.

\paragraph{The formal setting}
We recall some notation and results from \cite{Baake-Jaeger-Lenz2015}, where a cut-and-project scheme is associated with a Toeplitz sequence $\eta\in\{0,1\}^\Z$.
\begin{compactitem}[  -]
\item $p_1\mid p_2\mid p_3 \dots\to\infty$ is a period structure of $\eta$.
\item $\cP_n^i:=\{k\in\Z: \eta_{|k+p_n\Z}=i\}$ denotes $p_n$-periodic positions of $i=0,1$ on $\eta$ and $\cH_n:=\Z\setminus(\cP_n^0\cup\cP_n^1)$ denotes the set of \emph{holes at level $n$}.
\item $G:=\varprojlim\Z/p_n\Z$ and $\Delta\colon \Z \rightarrow G$, where $\Delta(n) = (n+ p_1\Z,n+ p_2\Z,\ldots)$ denotes the diagonal embedding.
\item $T\colon G \rightarrow G$ denotes the rotation by $\Delta(1)$, i.e. $(Tg)_n = g_n + 1+p_n\Z$ for all $n\in\N$.
\item The topology on $G$ is generated by the (open and closed) cylinder sets
$$U_n(h) := \{g\in  G \ : \   g_n = h_n\},\; h\in G.$$
\item $V^i:=\bigcup_{n\in\N}\bigcup_{k\in\cP_n^i\cap[0,p_n)}U_n(\Delta(k))$\; for $i=0,1$.
\item The \emph{window} $W:=\overline{V^1}=G\setminus V^0$ is topologically regular, i.e.~ $\overline{\inn(W)}=W$.
\item $\phi\colon G\rightarrow \{0,1\}^\Z$
is the coding function: $(\phi(g))_n = \mathbf{1}_W(g+\Delta(n))$.
Observe that $\phi(\Delta(0))=\eta$.
\end{compactitem}

\paragraph{More background material and an outline of the paper}

The odometer $(G,T)$ is the MEF of $(X_{\eta},\sigma)$, see \cite{MR756807}.
Let $F\colon X_\eta \to X_\eta$ be an automorphism commuting with $\sigma$, and
denote by $\pi:(X_\eta,\sigma)\to(G,T)$ that factor map onto the MEF which is uniquely determined by $\pi(\eta)=\Delta(0)$. Then there is $y_F\in G$ such that the rotation $f:y\mapsto y+y_F$ on $G$ represents $F$ in the sense that $\pi\circ F=f\circ\pi$ (see e.g.~\cite[Lemma~2.4]{Donoso2015}).
Observe that $y_F=\pi(F(\eta))$.

Denote by $C_\phi$ the set of continuity points of $\phi:G\to\{0,1\}^\Z$. Then $\card{\pi^{-1}\{\pi(x)\}}=1$
if and only if $\pi(x)\in C_\phi$, and in this case $x=\phi(\pi(x))$. This is folklore knowledge, but the reader may consult \cite[Remark~4.2]{KR2018} and also \cite[Sec.~5--7]{Downarowicz2005} for a related and more general perspective on this point.

Since $F$ is a bijection respecting the fibre structure $X_\eta=\bigcup_{h\in G}\pi^{-1}\{h\}$, we have $\card{\pi^{-1}\{\pi(F(x))\}}=\card{\pi^{-1}\{\pi(x)\}}$ for all $x\in X_\eta$.
In particular $f(C_\phi)=C_\phi$, i.e.
$C_\phi+y_F=C_\phi$,
see also \cite[Lemma~4.2]{Downarowicz2005}. As the set of discontinuities of the indicator function $\mathbf{1}_W$ is precisely the boundary $\partial W$, a moment's reflection shows that $X_\eta\setminus C_\phi=\bigcup_{k\in\Z}\partial W+\Delta(k)$, so
\begin{equation}\label{eq:partialW-invariance-1}
\partial W+y_F\subseteq \bigcup_{k\in\Z}\partial W+\Delta(k).
\end{equation}
We will use only this property of $y_F$ to investigate the nature of possible automorphisms~$F$. It is quite natural to expect that this will be much facilitated if the union on the r.h.s.~of \eqref{eq:partialW-invariance-1} is disjoint. Indeed, for general Toeplitz sequences, Bu\l{}atek and Kwiatkowski \cite{Bulatek1990} studied the centralizer problem under this assumption, because
their condition {\Sh} is equivalent to the \emph{disjointness condition}
\begin{equation}\tag{D}\label{eq:D}
\partial W\cap(\partial W+\Delta(k))=\emptyset\text{\; for all }k\in\Z\setminus\{0\}.
\end{equation}
Namely, condition {\Sh} is satisfied if and only if each $T$-orbit in $G$ hits the set of discontinuities of $\phi$ at most once, which is clearly equivalent to property \eqref{eq:D}, see also the remark after Definition~1 in \cite{DKL1995}.

As in \cite[Prop.~3]{Bulatek1990} it follows that
\begin{enumerate}[1)]
\item each fibre over a point in the MEF contains either exactly one or exactly two points, and
\item there exists $N\in\N$ such that $\partial W+y_F\subseteq\bigcup_{|k|\leqslant N}\partial W+\Delta(k)$.\label{eq:partialW-invariance-1-strong}
\end{enumerate}
We will introduce a weaker disjointness condition which also implies \ref{eq:partialW-invariance-1-strong}), see Proposition~\ref{prop:use(D')-new}:
\begin{equation}\tag{D'}\label{eq:D'}
\forall k\in\Z\setminus\{0\}:\ \partial W\cap(\partial W-\Delta(k))\text{ is nowhere dense w.r.t.~the subspace topology of }\partial W.
\end{equation}
Moreover we need a strengthened version of this condition which, in many examples, will help to show that $\partial W+y_F\subseteq\partial W+\Delta(k_0)$ for a single $k_0\in\Z$:
\begin{equation}\tag{DD'}\label{eq:DD'}
\begin{split}
&\hspace*{5cm}\forall k\in\Z\setminus\{0\}\ \forall \beta\in G:\\
&\partial W\cap(\partial W-\beta)\cap(\partial W-2\beta-\Delta(k))\text{ is nowhere dense w.r.t.~the subspace topology of }\partial W.
\end{split}
\end{equation}
Conditions \eqref{eq:D'} and \eqref{eq:DD'} and an additional growth restriction on the arithmetic structure, see \eqref{eq:theo2-ass-2-new-a}, play essential roles for proving that the assumption of our basic Theorem~\ref{theo:basic-case} is satisfied. To prove \eqref{eq:D'} and \eqref{eq:DD'}, we assume in \eqref{eq:additional-structure} below that the sets of essential holes have some particular arithmetic structure motivated by the intended applications to $\cB$-free Toeplitz shifts, see Propositions~\ref{prop:invariance-new},~\ref{prop:sufficient-for-(D')} and~\ref{prop:local-invariance-new-2}.
We note here that the additional growth restriction excludes irregular Toeplitz shifts, see Remark~\ref{rem:regular1}.
In Theorem~\ref{theo:ess-holes-arithmetic} we show that $\cB$-free Toeplitz subshifts indeed satisfy the structural assumption~\eqref{eq:additional-structure}.

Theorem~1 of \cite{Bulatek1990} provides a characterization of the triviality of the centralizer. The example $\cB_1=\{2^nc_n:n\in\N\}$ with coprime odd $c_n>1$ from \cite{Bartnicka2017} satisfies \eqref{eq:D}, but we were unable
to evaluate the criterion from \cite[Theorem~1]{Bulatek1990} for it. Instead we will show that the example not only satisfies
$\partial W+y_F\subseteq\bigcup_{|k|\leqslant N}\partial W+\Delta(k)$, but that there exists a single integer $k_0$ such that $\partial W+y_F\subseteq\partial W+\Delta(k_0)$, i.e.~ $\partial W+(y_F-\Delta(k_0))\subseteq\partial W$, see Example~\ref{ex:power:identity}.
The same holds for the example $\cB_1'=\cB_1\cup\{c_1^2\}$ and for further generalizations of this kind.
We shall see that this is the key to control the centralizer with modest efforts in Corollary~\ref{coro:simple-case}: In case of $\cB_1$ the centralizer is trivial (see also \cite{Bartnicka2017}),
while for $\cB_1'$ our approach only yields that the \mbox{$c_1$-th} iterate of each centralizer element is trivial. Indeed, we will show for this example that the centralizer has an element of order $c_1$, see Proposition~\ref{prop:non-trivial_fin_ord} and Remark~\ref{rem:finite-order}. More generally, we will show that elements of the centralizer can be of arbitrarily large finite order, see Corollary~\ref{coro:large-centralizer}.

Already the example $\cB_2=\{2^nc_n, 3^nc_n:n\in\N\}$, with coprime  $c_n>1$ also coprime to $2$ and $3$ and satisfying $\prod_{n\in\N}\left(1-\frac{1}{c_n}\right)>\frac{1}{2}$, violates condition \eqref{Sh} but satisfies conditions \eqref{eq:D'} and \eqref{eq:DD'}, see Example~\ref{ex:2^i3^i}.
In Section~\ref{sec:complexity} we show that the shift determined by $\cB_2$ has superpolynomial complexity (the same holds for $\cB_1$), so that the results from \cite{MR3324942, MR3430839, Donoso2015} do not apply. We also deliver examples for which not all holes are essential, see Examples~\ref{ex:not_all_holes_new} and~\ref{ex:not-always-GH}.
\section{The abstract regular Toeplitz case}\label{sec:abstract_Toeplitz}
\subsection{A first basic theorem}
Let $\eta\in\{0,1\}^\Z$ be a non-periodic Toeplitz sequence with period structure
$p_1\mid p_2\mid p_3 \dots\to\infty$.  That means for each $k\in\Z$ there exists $n\ge1$ such that $\eta_{|k+p_n\Z}$ is constant, but $\eta$ is not periodic.
For $i=0,1$ denote
$\cP_n^i=\{k\in\Z: \eta_{|k+p_n\Z}=i\}$\footnote{We use this notation, because it is shorter than $\operatorname{Per}_{p_n}(\eta,i)$ established in the literature.} and $\cH_n=\Z\setminus (\cP_n^0\cup\cP_n^1)$.
$\cH_n$ is called the set of \emph{holes at level $n$}.

The odometer group
$G=\varprojlim\Z/p_n\Z$ with the $\Z$-action  ``addition of $1$'' ($T:G\to G$, $(Tg)_n=g_n+1$) is the MEF of the subshift $X_\eta$, which is the orbit closure of $\eta$ under the left shift $\sigma$ on $\{0,1\}^\Z$. In symbols: $\pi:(X_\eta,\sigma)\to(G,T)$. Observe that $\cH_n\subseteq\cH_N$ if $n>N$.

In \cite{Baake-Jaeger-Lenz2015}, a cut and project scheme is associated with $\eta$ by specifying a compact and topologically regular window $W\subseteq G$:
For $h\in G$ let $U_n(h)=\{g\in G: g_n=h_n\}$ and define
\begin{equation*}
V^i=\bigcup_{n\in\N}\bigcup_{k\in\cP_n^i\cap[0,p_n)}U_n(\Delta(k))
\end{equation*}
for $i=0,1$. From \cite[Theorem~1 and its proof]{Baake-Jaeger-Lenz2015} we see:
\begin{compactenum}
\item $W:=\overline{V^1}=G\setminus V^0$ is topologically regular, i.e.~ $\overline{\inn(W)}=W$,
\item $\partial W=G\setminus(V^0\cup V^1)$,
\item $\Delta(k)\in V^i$ iff $\eta_k=i$ for $i=0,1$ and all $k\in\Z$, in particular
$\eta_k=1$ iff $\Delta(k)\in W$.
\end{compactenum}

\begin{lemma}\label{lemma:newG-1}
$h\in\partial W$ if and only if $h_N\in\cH_N$ for all $N\in\N$.
\end{lemma}
\begin{proof}
\begin{equation*}
\begin{split}
h\not\in\partial W
&\Leftrightarrow
h\in\bigcup_{i\in\{0,1\}}V^i
\Leftrightarrow
h\in\bigcup_{i\in\{0,1\}}\bigcup_{n\in\N}\bigcup_{k\in\cP_n^i\cap[0,p_n)}U_n(\Delta(k))\\
&\Leftrightarrow
\exists i\in\{0,1\}\ \exists n\in\N\ \exists k\in\cP_n^i\cap[0,p_n): h_n=k\\
&\Leftrightarrow
\exists i\in\{0,1\}\ \exists n\in\N: h_n\in\cP_n^i\\
&\Leftrightarrow
\exists n\in\N: h_n\not\in \cH_n
\end{split}
\end{equation*}
\end{proof}
\begin{definition}[Essential holes]\label{def:essential-holes}
The set of \emph{essential holes at level $N$} is defined as
\begin{equation}\label{eq:wcH_n}
\wcH_N:=\{k\in\cH_N: \cH_n\cap(k+p_N\Z)\neq\emptyset\text{ for all }n\ge N\}.
\end{equation}
\end{definition}
\begin{definition}\label{def:q_n}
The minimal periods of $\cH_n$ and $\wcH_n$ are denoted by $\tau_n$ and $\wq_n$, respectively.
\end{definition}
\begin{remark}
\begin{compactenum}[a)]
\item $\wcH_n\subseteq\wcH_N$ if $n>N$.
\item $\cH_N$ and $\wcH_N$ are $p_N$-periodic by definition - although this need not be their minimal period. (For $\wcH_N$ just observe that $k\in\wcH_N$ if and only if $k+p_N\in\wcH_N$.) Hence $\tau_N\mid p_N$ and $\wq_N\mid p_N$, so that expressions like ``$\tau_N\mid h_N$'' or ``$\gcd(\wq_N,h_N)$'' are well defined, when $h_N\in\Z/p_N\Z$.\footnote{This is consistent with the following general convention: an element $z$ of an abelian group $Z$ is divisible by $n\in\N$ if 
$z\in nZ$. If $Z$ is a cyclic group and $z_1,z_2\in Z$ then $\gcd(z_1,z_2)$ is a generator of the subgroup of $Z$ generated by $z_1$ and $z_2$. If $Z=\Z$ then we choose $\gcd(z_1,z_2)$ to be positive by convention. Finally, if $n\in \Z$ and $z+p\Z\in\Z/p\Z$, then we understand by $\gcd(n,z+p\Z)$ the (positive) generator of the group generated by $n$ and $z+p\Z$; this is the greatest common divisor of the numbers $n,z$ and $p$ in the usual sense.}
\end{compactenum}
\end{remark}

\begin{lemma}\label{lemma:newG-2}
\begin{compactenum}[a)]
\item $k\in\wcH_N$ if and only if $U_N(\Delta(k))\cap\partial W\neq\emptyset$.
\item $h\in\partial W$ if and only if $h_N\in\wcH_N$ for all $N\in\N$.
\end{compactenum}
\end{lemma}

\begin{proof}
a)\; Let $h\in U_N(\Delta(k))\cap\partial W$. Then $k\in h_N+p_N\Z\subseteq\cH_N+p_N\Z=\cH_N$ by Lemma~\ref{lemma:newG-1}, and for all $n\ge N$ we have in view of Lemma~\ref{lemma:newG-1}: $h_n\in\cH_n\cap(h_N+p_N\Z)=\cH_n\cap(k+p_N\Z)$. For the reverse implication let $k\in\wcH_N$. We construct $h\in\partial W$ with $h_N=k\mod p_N$: for $j>N$ there is $r_j\in\cH_j\cap(k+p_N\Z)$. Let $h^{(j)}=\Delta(r_j)$ and fix any accumulation point $h$ of the sequence $(h^{(j)})_j$.
Consider any $n\ge N$. For some sufficiently large $j_n\ge n$ we have $h_n=h_n^{(j_n)}=r_{j_n}\mod p_n$.
Hence $h_N\in r_{j_N}+p_N\Z=k+p_N\Z$, so that $h\in U_N(\Delta(k))$, and $h_n\in r_{j_n}+p_n\Z\subseteq\cH_{j_n}+p_n\Z\subseteq \cH_n+p_n\Z=\cH_n$, so that $h\in\partial W$ by Lemma~\ref{lemma:newG-1}.\\
b)\;
\begin{equation*}
\begin{split}
h\in\partial W
&\Leftrightarrow
\forall N\in\N:\ U_N(h)\cap\partial W\neq\emptyset\\
&\Leftrightarrow
\forall N\in\N:\ U_N(\Delta(h_N))\cap\partial W\neq\emptyset\\
&\Leftrightarrow
\forall N\in\N:\ h_N\in\wcH_N\quad\text{by part a).}
\end{split}
\end{equation*}
\end{proof}
The following simple lemma is basic for our approach.
\begin{lemma}\label{lemma:basic-new}
\begin{compactenum}[a)]
\item If $h+\beta\Z\subseteq\partial W$ for some $h,\beta\in G$, then
$h_n+\gcd(\beta_n,\wq_n)\Z\subseteq\wcH_n$ for all $n>0$.
\item
If $\partial W+\beta\subseteq\partial W$ for some $\beta\in G$, then
$\wq_n\mid\beta_n$ for all $n>0$.
\end{compactenum}
\end{lemma}

\begin{proof}
a)\;$h_n+\beta_n\Z\subseteq\wcH_n$ for all $n$ by Lemma~\ref{lemma:newG-2}b).
As $\wcH_n$ is $\wq_n$-periodic, this implies $h_n+\gcd(\beta_n,\wq_n)\Z\subseteq\wcH_n$.
\\
b)\;Let $j\in\wcH_n$. Then there is some $h\in U_n(\Delta(j))\cap\partial W$ by Lemma~\ref{lemma:newG-2}a). Now $h+\beta\Z\subseteq\partial W$ by assumption. As $h_n=j$, this implies, in view of part a) of the lemma, $j+\gcd(\beta_n,\wq_n)\Z\subseteq\wcH_n$. As this holds for each $j\in\wcH_n$, we have $\wcH_n+\gcd(\beta_n,\wq_n)\Z=\wcH_n$, and as $\wq_n$ is the minimal period of $\wcH_n$, we conclude that $\wq_n\mid\beta_n$.
\end{proof}

\begin{remark}\label{rem:regular}
If $h+\beta\Z\subseteq\partial W$, the lemma implies
\begin{equation*}
\ddelta(\cH_n)\ge \ddelta(\wcH_n)\ge \frac{1}{\gcd(\beta_n,\wq_n)}.
\end{equation*}
Therefore the last lemma seems to be useful only for regular Toeplitz shifts, because for irregular ones $\inf_n\ddelta(\cH_n)>0$, so that no useful lower bound on $\gcd(\beta_n,\wq_n)$ can be expected.
\end{remark}

Each automorphism $F$ of $(X_\eta,\sigma)$ determines an element $y_F\in G$ such that
$\pi(F(h))=\pi(h)+y_F$ for all $h\in G$. Observe that
\begin{equation}\label{eq:partialW-invariance-1-new}
\partial W+y_F\subseteq\bigcup_{k\in\Z}\left(\partial W+\Delta(k)\right),
\end{equation}
because $F$ leaves the set of non-one-point fibres over the MEF invariant.
We first focus on a stronger property than \eqref{eq:partialW-invariance-1-new}.
\begin{theorem}\label{theo:basic-case}
Recall that $\wq_n$ denotes the minimal period of $\wcH_n$.
\begin{compactenum}[a)]
\item If an automorphism $F$ of $(X_\eta,\sigma)$ satisfies $\partial W+y_F\subseteq\partial W+\Delta(k)$ for some $k=k_F\in\Z$, then $\wq_n\mid(y_F)_n-k_F$. In particular, if infinitely many $\wcH_n$ have minimal period $p_n$, then the centralizer of $(X_\eta,\sigma)$ is trivial.
\item Suppose that for each automorphism $F\in\Aut_\sigma(X_\eta)$ there exists a unique $k_F\in\Z$ such that
\begin{equation}\label{eq:unique_k}
\partial W+y_F\subseteq \partial W+\Delta(k_F).
\end{equation}
If $M:=\liminf_{n\to\infty}p_n/\wq_n<\infty$, then \[\Aut_\sigma(X_\eta)=\langle \sigma\rangle\oplus\Tor,\] where $\Tor$ denotes the torsion group of $\Aut_\sigma(X_\eta)$. Moreover, $\Tor$ is a cyclic group (possibly trivial), whose order divides $M$.
\end{compactenum}
\end{theorem}

\begin{proof}
a)\; The first claim follows from Lemma~\ref{lemma:basic-new}b), the second one is just a special case of this.\\
b)\; In each residue class of $\Aut_\sigma(X_\eta)/\langle \sigma\rangle$ there is exactly one element $F$ for which the associated integer $k_F$ satisfying \eqref{eq:unique_k} equals $0$. These elements $F$ form a subgroup $J$ of $\Aut_\sigma(X_\eta)$.
Suppose for a contradiction that there are $M+1$ different automorphisms $F_1,\dots,F_{M+1}\in J$.
In view of Lemma~\ref{lemma:basic-new}b) they all satisfy
\begin{equation}\label{eq:divisibility-new}
\wq_n\mid (y_{F_{i}})_{S_n} \;\text{ for all }n\in\N
\end{equation}
Hence there exists arbitrarily large $n\in\N$ such that
$M=p_n/\wq_n$ and
\begin{equation*}
(y_{F_i})_{S_n}/\wq_n\in\{0,\dots,M-1\} \text{ for all $i=1,\dots,M+1$.}
\end{equation*}
 It follows that there exist two different $i,j\in\{1,\dots,M+1\}$ for which $(y_{F_i})_{S_n}=(y_{F_j})_{S_n}$ for infinitely many $n$, which is only possible if $y_{F_i}=y_{F_j}$.
In view of
\cite[Lemma~2.4]{Donoso2015} this implies $F_i=F_j\mod\langle \sigma\rangle$, so that $i=j$,
a contradiction.
Hence $J$ is a finite group
of order  $m\leqslant M$, say $J=\{F_1,\dots,F_m\}$. In particular, $J\subseteq\Tor$.
On the other hand, if $F\in\Tor$, then $ry_F=\Delta(0)\in G$ for some positive integer~$r$. Hence, with the integer $k_F$ satisfying \eqref{eq:unique_k}, we have
\[
\partial W+y_{Id}-\Delta(rk_F)=\partial W-\Delta(rk_F)=\partial W+ry_F-\Delta(rk_F)\subseteq \partial W,
\]
but $k_{Id}=0$, so \eqref{eq:unique_k} implies $k_F=0$. It follows that $F\in J$, and we proved that $J=\Tor$. Theorem~3.2(2) of \cite{Donoso2017} then implies that $\Tor$ is cyclic, and $\Aut_\sigma(X_\eta)=J\oplus\langle \sigma\rangle=\Tor\oplus\ \langle \sigma\rangle$.

It remains to determine the order $m$ of $\Tor$:
Fix some $F\in\Tor$.
In view of \eqref{eq:divisibility-new},
\begin{equation*}
p_n=M\cdot\wq_n\mid M\cdot(y_F)_{S_n}
\equiv (y_{F^M})_{S_n}\mod p_n
\end{equation*}
for all $n\in\N$.
Hence $F^M=\id_{X_\eta}$, so that the order of $F$ is a divisor of $M$.
\end{proof}

In order to apply this theorem we need to verify assumption \eqref{eq:unique_k} and to determine the periods $\wq_n$.
So we focus next on finding sufficient conditions that imply
$\partial W+(y_F-\Delta(k))\Z\subseteq\partial W$.

\subsection{The separated holes conditions and its variants}
\label{subsec:separated-holes-contition}

The \emph{Separated holes condition} (Sh) was introduced in \cite{Bulatek1990}. It requires
\begin{equation}\tag{Sh}\label{Sh}
\forall k\in\Z\setminus\{0\}\ \exists N\in\N\ \forall n\ge N: \cH_n\cap(\cH_n-k)=\emptyset.
\end{equation}
As mentioned in the introduction, it is equivalent to the {\em Disjointness condition}
\begin{equation}\tag{D}\label{D}
\forall\ k\in\Z\setminus\{0\}: \partial W\cap (\partial W-\Delta(k))=\emptyset.
\end{equation}
Indeed, it is even equivalent to the {\em Separated essential holes condition}
\begin{equation}\tag{Seh}\label{Seh}
\forall k\in\Z\setminus\{0\}\ \exists N\in\N\ \forall n\ge N: \wcH_n\cap(\wcH_n-k)=\emptyset.
\end{equation}
As we do not make use of this equivalence, we leave it as an exercise.

The following variants of the Separated essential holes condition, which allow to study Toeplitz subshifts that violate \eqref{Sh}, will play important roles, however, so we provide proofs for the corresponding equivalences:
\begin{compactitem}[-]
\item {\em weak Disjointness condition} \eqref{D'}, equivalent to {\em weak Separated essential holes condition} \eqref{Seh'}:
\begin{equation}\tag{D'}\label{D'}
\forall k\in\Z\setminus\{0\}:\ \partial W\cap(\partial W-\Delta(k))\text{ is nowhere dense w.r.t.~the subspace topology of }\partial W.
\end{equation}
\begin{equation}\tag{Seh'}\label{Seh'}
\begin{split}
&\forall k\in\Z\setminus\{0\}:\
\text{there is no arithmetic progression }r+p_N\Z\text{ such that}\\
&\hspace*{0.7cm}\forall n\ge N:\ \emptyset\neq(r+p_N\Z)\cap\wcH_n\subseteq\wcH_n-k.
\end{split}
\end{equation}
\item {\em weak Double Disjointness condition} \eqref{DD'}, equivalent to {\em weak Double Separated essential holes condition} \eqref{DSeh'}:
\begin{equation}\tag{DD'}\label{DD'}
\begin{split}
&\hspace*{5cm}\forall k\in\Z\setminus\{0\}\ \forall \beta\in G:\\
&\partial W\cap(\partial W-\beta)\cap(\partial W-2\beta-\Delta(k))\text{ is nowhere dense w.r.t.~the subspace topology of }\partial W.
\end{split}
\end{equation}
\begin{equation}\tag{DSeh'}\label{DSeh'}
\begin{split}
&\forall k\in\Z\setminus\{0\}\ \forall \beta\in G:\
\text{there is no arithmetic progression }r+p_N\Z\text{ such that}\\
&\hspace*{0.7cm}\forall n\ge N:\ \emptyset\neq(r+p_N\Z)\cap\wcH_n\subseteq (\wcH_n-\beta_n)\cap(\wcH_n-2\beta_n-k).
\end{split}
\end{equation}
\end{compactitem}
Observe that for $\beta=0$ conditions \eqref{DD'} and \eqref{DSeh'} reduce to \eqref{D'} and \eqref{Seh'}, respectively. Moreover, conditions~\eqref{D} and \eqref{Sh} clearly imply \eqref{D'} and \eqref{Seh'}, respectively.

In the sequel we will assume \eqref{DD'} -- indeed, for some results only the weaker \eqref{D'} is needed. In the $\cB$-free setting, \eqref{DD'} will be verified under suitable assumptions in Proposition~\ref{prop:local-invariance-new-2}.

Both equivalences above, namely $\lnot$\eqref{D'} $\Leftrightarrow$ $\lnot$\eqref{Seh'} and $\lnot$\eqref{DD'} $\Leftrightarrow$ $\lnot$\eqref{DSeh'}, follow immediately from the next lemma.

\begin{lemma}\label{lemma:weak-separated-essential-holes-contition}
Let $k\in\Z\setminus\{0\}$, $\beta\in G$, $N>0$ and $r\in\Z$. Then
\begin{equation}\label{eq:non-DD'}
\emptyset\neq U_N(\Delta(r))\cap\partial W\subseteq(\partial W-\beta)\cap(\partial W-2\beta-\Delta(k))
\end{equation}
if and only if
\begin{equation}\label{eq:non-DSeh'}
\forall n\ge N:\
\emptyset\neq (r+p_N\Z)\cap\wcH_n\subseteq(\wcH_n-\beta_n)\cap(\wcH_n-2\beta_n-k)
\end{equation}
\end{lemma}

\begin{proof}
Suppose that \eqref{eq:non-DD'} holds. Then, for all $n\ge N$, there is $r_n\in r+p_N\Z$ such that $U_n(\Delta(r_n))\cap\partial W\neq\emptyset$, whence $r_n\in(r+p_N\Z)\cap\wcH_n$ in view of Lemma~\ref{lemma:newG-2}a).

Now consider any
$r_n\in(r+p_N\Z)\cap\wcH_n$. By the inclusion in~\eqref{eq:non-DD'},
\begin{equation*}
U_n(\Delta(r_n))\cap\partial W\subseteq U_N(\Delta(r_n))\cap\partial W
=U_N(\Delta(r))\cap\partial W\subseteq(\partial W-\beta)\cap(\partial W-2\beta-\Delta(k)),
\end{equation*}
so that, by Lemma~\ref{lemma:newG-2}a) again,
$r_n+\beta_n\in\wcH_n$ and $r_n+2\beta_n+k\in\wcH_n$, which proves~\eqref{eq:non-DSeh'}.

Conversely, suppose that \eqref{eq:non-DSeh'} holds. For each $n\ge N$ there is some $r_n\in(r+p_N\Z)\cap\wcH_n$, and we find a subsequence $\Delta(r_{n_i})$ that converges to some $h\in G$. Hence, for each $m>0$ there is $n_i\ge m$ such that $h_m\in r_{n_i}+p_m\Z\subseteq \wcH_{n_i}+p_m\Z\subseteq\wcH_m+p_m\Z=\wcH_m$, so that
$h\in\partial W$ in view of Lemma~\ref{lemma:newG-2}b). As $h_N\in r_{n_i}+p_N\Z=r+p_N\Z$ for some $n_i$, this shows that $h\in U_N(\Delta(r))\cap\partial W$.

Now consider any $h\in U_N(\Delta(r))\cap\partial W$. Then, for all $n\ge N$, $U_n(h)\cap\partial W\neq\emptyset$, so that $h_n\in (r+p_N\Z)\cap\wcH_n$, where we used
Lemma~\ref{lemma:newG-2}a) once more. The inclusion in~\eqref{eq:non-DSeh'}
then implies $h_n\in (\wcH_n-\beta_n)\cap(\wcH_n-2\beta_n-k)$ for all $n\ge N$. Now
Lemma~\ref{lemma:newG-2}b) shows that $h+\beta\in\partial W$ and $h+2\beta+\Delta(k)\in\partial W$, which proves~\eqref{eq:non-DD'}.
\end{proof}
\begin{remark}\label{remark:condition(*)}
Suppose that the condition (*) from \cite{MR1199322} holds,
i.e. for any $s\in\Z$
\begin{equation*}
[sp_n,(s+1)p_n)\cap\cH_n\subset\cH_{n+1}\text{ or }[sp_n,(s+1)p_n)\cap\cH_{n+1}=\emptyset,
\end{equation*}
and
recall that it implies triviality of the centralizer~\cite[Theorem 1]{MR1199322}.  Since we study only non-periodic Toeplitz sequences, for any $n\in\N$ there exists some $s\in\Z$ such that $[sp_n,(s+1)p_n)\cap\cH_n\subset\cH_{n+1}$. So $\wcH_n=\cH_n$ for any $n\geq1$. A simple inductive reasoning shows that condition (*) implies: for any $n\ge N$, any $r,r'\in\cH_N\cap[0,p_N)$ and any $s\in\Z$ we have
\begin{equation*}
r+sp_N\in\cH_n\Leftrightarrow r'+sp_N\in\cH_n\Leftrightarrow r+sp_N\in\cH_n-(r'-r),
\end{equation*}
in particular $(r+p_N\Z)\cap \cH_n\subseteq\cH_n-(r'-r)=\wcH_n-(r'-r)$.
Hence, if $[0,p_N)$ contains at least two holes at level $N$, then \eqref{Seh'} does not hold. If, on the other hand, $\cH_n\cap[0,p_n)$ is a singleton for any $n\geq N$, then the distance between consecutive holes at level $n$ is $p_n$, so even~\eqref{Sh} holds.
\end{remark}

\begin{remark}\label{remark:holes_periods}
Given $k\in\cH_N\setminus \wcH_N$, let $n_k\ge N$ be  minimal such that $(k+p_N\Z)\cap \cH_{n_k}=\emptyset$. Clearly, $n_k$ depends only on the residue of $k$ modulo $p_N$, so the $n_k$ are bounded, say, by $m_N$. Then, for $k\in \cH_N$, $k\in \wcH_N$ if and only if $(k+p_N\Z)\cap \cH_{m_N}\neq\emptyset$. More generally, for every $n\ge m_N$:  $k\in \wcH_N$ if and only if $(k+p_N\Z)\cap \cH_{n}\neq\emptyset$. Hence, for every $n\ge m_N$, $\wcH_N=\cH_{n}+p_N\Z$. It follows that the minimal period $\wq_N$ of $\wcH_N$ divides $\gcd(\tau_n,p_N)$ for $n\ge m_N$.
\end{remark}

\subsection{Consequences of the weak Disjointness condition \eqref{D'}}

Consider any automorphism $F$ of $(X_\eta,\sigma)$. Recall from the introduction that
$\pi\circ F=f\circ\pi$, where $\pi:(X_\eta,\sigma)\to(G,T)$ is the MEF-map and $f:G\to G, y\mapsto y+y_F$ for some $y_F\in G$,
and that $\partial W+y_F\subseteq\bigcup_{k\in\Z}\partial W+\Delta(k)$. Denote
\begin{equation}
V_k=\partial W\cap\left(\partial W+\Delta(k)-y_F\right)\quad(k\in\Z)
\end{equation}
and
\begin{equation}\label{eq:K-def-new}
K=\{k\in\Z: \inn_{\partial W}(V_k)\neq\emptyset\}.
\end{equation}

\begin{proposition}\label{prop:use(D')-new}
Assume that the weak disjointness condition~\eqref{D'} holds.
Let the automorphism $F$ of~$(X_\eta,\sigma)$ be described by a block code $\{0,1\}^{[-m:m]}\to\{0,1\}$. Then
\begin{equation}\label{eq:finite-cover-1-new}
\partial W+y_F\subseteq\bigcup_{k=-m}^m\partial W+\Delta(k).
\end{equation}
\end{proposition}
\begin{proof}
Recall from \eqref{eq:partialW-invariance-1-new} that
$\partial W+y_F
\subseteq
\bigcup_{k\in\Z}\partial W+\Delta(k)$.

Let $y\in G$ and recall that $\pi\colon {X_\eta\to G}$ denotes the factor map onto the MEF. At the end of the proof we show
\begin{equation}\label{eq:card=2-new}
\card{\pi^{-1}\{y\}}=2
\quad\Leftrightarrow\quad
\exists j\in\Z:\ y+\Delta(j)\in R:=\partial W\setminus\bigcup_{k\in\Z\setminus\{0\}}\partial W+\Delta(k).
\end{equation}
As $F\colon X_\eta\to X_\eta$ is a bijection that maps $\pi$-fibres to $\pi$-fibres and as $\pi\circ F=f\circ\pi$, it follows that
when $\pi^{-1}\{y\}=\{x_1,x_2\}$ with $x_1\neq x_2$, then $\pi^{-1}\{f(y)\}=\{x_1'=F(x_1),x_2'=F(x_2)\}$, and
there are exactly one index $j\in\Z$ such that $y+\Delta(j)\in\partial W$ and $(x_1)_j\neq (x_2)_j$, and exactly one index $k\in\Z$ such that $f(y)+\Delta(k)\in\partial W$ and $(x_1')_{k}\neq (x_2')_{k}$.
As $F$ is described by a block code $\{0,1\}^{[-m:m]}\to\{0,1\}$, it follows that $|j-k|\leqslant m$. \footnote{This argument is taken from the proof of \cite[Cor.~1]{Bulatek1990}.}

Consider any $y\in R$. Then the index $j$ in \eqref{eq:card=2-new} equals $0$ and
$y+y_F\in \partial W-\Delta(k)$ for some $k$ with $|k|\leqslant m$. In other words:
$f(R)$ is contained in the closed set $\bigcup_{k=-m}^mT^k(\partial W)$.
Because of \eqref{D'}, the set $R$ defined in \eqref{eq:card=2-new} is  residual w.r.t.~the subspace topology of $\partial W$.
Hence
\begin{equation*}
\partial W+y_F=f(\partial W)=f(\overline{R})\subseteq\overline{f(R)}\subseteq
\bigcup_{k=-m}^mT^k(\partial W)
=\bigcup_{k=-m}^m\partial W+\Delta(k).
\end{equation*}

It remains to prove \eqref{eq:card=2-new}.
Recall first that $\card{\pi^{-1}\{y\}}=1$ if and only if $y\in C_\phi$, so that
$\card{\pi^{-1}\{y\}}>1$ if and only if $y\in \bigcup_{k\in\Z}\partial W+\Delta(k)$.
So all we must show is that (A)\; $\card{\pi^{-1}\{y\}}>2$ if and only if (B)\; $y+\Delta(j)\in \partial W\cap (\partial W+\Delta(k))$ for some $j\in\Z$ and $k\in\Z\setminus\{0\}$.

Suppose first that (A) holds, i.e.~that there are at least three different points in $\pi^{-1}\{y\}$. As points in the same $\pi$-fibre can differ only at positions $k$ where $y+\Delta(k)\in\partial W$, there must be at least two such positions, and that is (B).

Conversely, if (B) holds and if $x=\phi(y)$, then $x_j=x_{j-k}=1$.
\begin{compactitem}[-]
\item As
$\partial W\cup (\partial W+\Delta(k))$ is nowhere dense in $G$,
there are arbitrarily small perturbations $y'$
of $y$ such that $y'+\Delta(j)\not\in \partial W\cup (\partial W+\Delta(k))$,
resulting in points $x'=\phi(y')$ with $x'_j=x'_{j-k}=0$.
\item As $\partial W\cap (\partial W+\Delta(k))$ is nowhere dense in $\partial W$ w.r.t.~the subspace topology on $\partial W$ (because of \eqref{D'}),
there are arbitrarily small perturbations $y'$ of $y$ such that $y'+\Delta(j)\in\partial W\setminus (\partial W+\Delta(k))$,
 resulting in $x'=\phi(y')$ with $x'_j=1$ and $x'_{j-k}=0$.
\end{compactitem}
Hence $\card{\pi^{-1}\{y\}}\geqslant 3$, and that is (A).
\end{proof}

\begin{proposition}\label{prop:Baire-new}
We have $K\neq\emptyset$, and
there is an at most countable collection $U_{{N_j}}(\Delta(r_j))$, ${j\in\N}$,  of cylinder sets in $G$ with the following properties:
For each $j\in\N$ there exists some $k_j\in\Z$ such that
\begin{equation}\label{eq:Baire-1-new}
\emptyset\neq\left(U_{{N_j}}(\Delta(r_j))\cap\partial W\right)+(y_F-\Delta(k_j))\subseteq\partial W,
\end{equation}
and
\begin{equation}\label{eq:Baire-2-new}
\bigcup_{j\in\N}U_{{N_j}}(\Delta(r_j))\cap\partial W\quad \text{ is dense in $\partial W$.}
\end{equation}
\end{proposition}
\begin{proof}
We start from observation~\eqref{eq:partialW-invariance-1-new}, namely
$\partial W+y_F\subseteq \bigcup_{k\in\Z}\partial W+\Delta(k)$.
This implies
\begin{equation*}
\partial W= \bigcup_{k\in\Z}\partial W\cap\left(\partial W+\Delta(k)-y_F\right)
=\bigcup_{k\in \Z}V_k.
\end{equation*}
Hence $M:=\bigcup_{k\in \Z}V_k\setminus\inn_{\partial W}(V_k)$ is a meagre subset of the compact space $\partial W$ and $\partial W=M\cup\bigcup_{k\in K}\inn_{\partial W}(V_k)$. Now Baire's category theorem implies that $K\neq\emptyset$. As $\partial W$ is separable, there is
an at most countable collection $U_{{N_j}}(\Delta(r_j))$, ${j\in\N}$, of cylinder sets in $G$, for each of which there exists $k_j\in K$ such that $\emptyset\neq\partial W\cap U_{{N_j}}(\Delta(r_j))\subseteq\inn_{\partial W}(V_{k_j})$, and such that
\begin{equation*}
\partial W=M\cup\bigcup_{j\in\N}\left(\partial W\cap U_{{N_j}}(\Delta(r_j))\right).
\end{equation*}
As $M$ is meagre in $\partial W$, these cylinder sets satisfy \eqref{eq:Baire-2-new}, and as
\begin{equation*}
\left(U_{{N_j}}(\Delta(r_j))\cap\partial W\right)
\subseteq
\inn_{\partial W}(V_{k_j})
\subseteq
\partial W+\Delta(k_j)-y_F,
\end{equation*}
also \eqref{eq:Baire-1-new} holds.
\end{proof}

\begin{corollary}\label{coro:new-1-new}
Assume that the weak disjointness condition~\eqref{D'} holds.
Let the automorphism $F$ of~$(X_\eta,\sigma)$ be described by a block code $\{0,1\}^{[-m:m]}\to\{0,1\}$. Then the set $K$ is contained in $[-m,m]$, $\inn_{\partial W}(V_{k_i})\cap\inn_{\partial W}(V_{k_j})=\emptyset$ for any different $k_i,k_j\in K$, and
$\partial W=\bigcup_{k\in K}V_k'$, where $V_k':=\overline{\inn_{\partial W}(V_k)}$.
\end{corollary}
\begin{proof}
Let $K':=K\cap[-m,m]$, $m$ as in Proposition~\ref{prop:use(D')-new}. Because of that proposition and Proposition~\ref{prop:Baire-new}, $\partial W=\bigcup_{k\in K'}\overline{\inn_{\partial W}(V_k)}$. Suppose there are $k_i,k_j\in K$ such that $\inn_{\partial W}(V_{k_i})\cap \inn_{\partial W}(V_{k_j})\neq\emptyset$. Then there is some cylinder set $U_{N}(\Delta(r))\cap\partial W$ contained in this intersection.
Let
$\tilde{U}:=(U_{N}(\Delta(r))\cap\partial W)+y_F-\Delta(k_i)$.
 Then $\tilde{U}\subseteq\partial W$ and $\tilde{U}+\Delta(k_i-k_j)\subseteq\partial W$, so that $\tilde{U}\subseteq \partial W\cap(\partial W+\Delta(k_j-k_i))$. As $\tilde{U}$ is nonempty and open in the relative topology on $\partial W$, the weak disjointness assumption implies $k_i=k_j$. This also proves that $K=K'\subseteq[-m,m]$.
\end{proof}

For later use we note a further consequence of Proposition~\ref{prop:Baire-new}:

\begin{corollary}\label{coro:new-2-new}
Assume that $\left(\inn_{\partial W}(V_{k_i})+(y_F-\Delta(k_i)\right)\cap \inn_{\partial W}(V_{k_j})\neq\emptyset$ for some $k_i,k_j\in K$. Then there exist $N\in\N$ (which can be chosen arbitrarily large) and $r\in\Z$ such that
\begin{equation}\label{eq:Baire-3-new}
\begin{split}
\emptyset\neq\left(U_{{N}}(\Delta(r))\cap\partial W\right)+(y_F-\Delta(k_i))\subseteq&\partial W\quad\text{and}\\
\emptyset\neq\left(U_{{N}}(\Delta(r))\cap\partial W\right)+(y_F-\Delta(k_i))+(y_F-\Delta(k_j))\subseteq&\partial W.
\end{split}
\end{equation}
\end{corollary}
\begin{proof}
Fix some cylinder set $U_{{N'}}(\Delta(r'))$ such that
\begin{equation*}
\emptyset\neq
U_{{N'}}(\Delta(r'))\cap\partial W\subseteq \inn_{\partial W}(V_{k_i})\cap \left(\inn_{\partial W}(V_{k_j})-(y_F-\Delta(k_i))\right).
\end{equation*}
In view of Proposition~\ref{prop:Baire-new} there are $N_i',N_j'\in\N$ and $r_i',r_j'\in\Z$ such that
\begin{equation}\label{eq:intersections-6-new}
U_{{N'}}(\Delta(r'))\cap U_{{N_i'}}(\Delta(r_i'))\cap \left(U_{{N_j'}}(\Delta(r_j'))-(y_F-\Delta(k_i))\right) \cap\partial W\neq\emptyset
\end{equation}
and
\begin{equation}\label{eq:intersections-7-new}
\left(U_{{N_i'}}(\Delta(r_i'))\cap\partial W\right)+(y_F-\Delta(k_i))\subseteq\partial W,\quad
\left(U_{{N_j'}}(\Delta(r_j'))\cap\partial W\right)+(y_F-\Delta(k_j))\subseteq\partial W.
\end{equation}
Because of \eqref{eq:intersections-6-new} and \eqref{eq:intersections-7-new}, the set
$U_{{N'}}(\Delta(r'))\cap U_{{N_i'}}(\Delta(r_i'))\cap \left(U_{{N_j'}}(\Delta(r_j'))-(y_F-\Delta(k_i))\right)$
contains a cylinder set $U_{N}(\Delta(r))$ for which $U_{N}(\Delta(r))\cap\partial W\neq\emptyset$ and
\begin{equation*}
\begin{split}
\left(U_{{N}}(\Delta(r))\cap\partial W\right)+(y_F-\Delta(k_i))\subseteq&\partial W\quad\text{and}\\
\left(U_{{N}}(\Delta(r))\cap\partial W\right)+(y_F-\Delta(k_i))+(y_F-\Delta(k_j))\subseteq&\partial W.
\end{split}
\end{equation*}
Clearly, $N$ can be chosen arbitrarily large.
\end{proof}

\subsection{Consequences of the weak Double Disjointness condition \eqref{DD'}}

The crucial step is now to show that $\inn_{\partial W}(V_k)+(y_F-\Delta(k))\subseteq V_k$ for all $k\in K$ under suitable assumptions.

\begin{proposition}\label{prop:local-invariance-new}
Assume that the weak double disjointness condition~\eqref{DD'} holds.
Let the automorphism $F$ of~$(X_\eta,\sigma)$ be described by a block code $\{0,1\}^{[-m:m]}\to\{0,1\}$.
Then $V_k'+(y_F-\Delta(k))\subseteq V_k'$ for all $k\in K$, the set $K$ is contained in $[-m,m]$, and the sets $V_k'$ have pairwise disjoint interiors.
\end{proposition}

\begin{proof}
It suffices to prove that
$\inn_{\partial W}(V_k)+(y_F-\Delta(k))\subseteq V_k'$ for all $k\in K$.
Suppose for a contradiction that there is $k_i\in K$ such that
$\inn_{\partial W}(V_{k_i})+(y_F-\Delta(k_i))\not\subseteq V_{k_i}'$.
By Proposition~\ref{prop:Baire-new}, there exists $k_j\in K\setminus\{k_i\}$ such that
$\left(\inn_{\partial W}(V_{k_i})+(y_F-\Delta(k_i))\right)\cap\inn_{\partial W}(V_{k_j})\neq\emptyset$.
So we can apply Corollary~\ref{coro:new-2-new}.
Hence there are $N\in\N$ and $r\in\Z$ such that, setting $\beta=y_F-\Delta(k_i)$,
\begin{equation*}
\left(U_N(\Delta(r))\cap \partial W\right)\subseteq\left(\partial W-\beta\right)\cap\left(\partial W-2\beta-\Delta(k_i-k_j)\right).
\end{equation*}
In view of condition \eqref{DD'}, this implies $k_i=k_j$.
The remaining assertions follow from Corollary~\ref{coro:new-1-new}.
\end{proof}

\begin{remark}\label{remark:Hnk}
Suppose that the conclusions of Proposition~\ref{prop:local-invariance-new} are satisfied (not necessarily condition~\eqref{DD'}).
\begin{compactenum}[a)]
\item For $k\in K$ let $\wcH_n^k:=\{j\in\wcH_n: U_n(\Delta(j))\cap V_k'\neq\emptyset\}$. Then $\wcH_n=\bigcup_{k\in K}\wcH_n^k$ by Lemma~\ref{lemma:newG-2} and Corollary~\ref{coro:new-1-new}. (But observe that this need not be a disjoint union, in general!)
\item $j\in\wcH_n^k$ if and only if $j+p_n\in\wcH_n^k$, i.e. all $\wcH_n^k$ are $p_n$-periodic.
\item For each $n>0$ and $k\in K$ we have $\wcH_n^k+\gcd((y_F)_n-k,p_n)\Z\subseteq\wcH_n^k$.
\begin{proof}
For each $j\in\wcH_n^k$ there exists $h\in U_n(\Delta(j))\cap V_k'$. By assumption, $h+(y_F-\Delta(k))\Z\subseteq V_k'$, so that, for all $t\in\Z$,
\begin{equation*}
U_n\left(\Delta(j+((y_F)_n-k)t\right)\cap V_k'=
U_n\left(\Delta(j)+(y_F-\Delta(k))t\right)\cap V_k'\neq\emptyset,
\end{equation*}
i.e. $j+((y_F)_n-k)\Z\subseteq\wcH_n^k$.
\end{proof}
\end{compactenum}
\end{remark}

\begin{remark}\label{claim}
Remark~\ref{remark:Hnk}c)
can be used to show that $\card K=1$ and hence $\partial W+ (y_F-\Delta(k))\subseteq\partial W$, whenever $\ddelta(\wcH_n)=o\left(\frac{1}{\sqrt p_n}\right)$.
This allows to apply Theorem~\ref{theo:basic-case}, which imposes restrictions on $y_F$ in terms of the minimal periods $\wq_n$ of the sets $\wcH_n$.
(But that is quite far from what holds in the $\cB$-free setting.)\\
{\em Proof} of the claim: Let $k,k'\in K$ and denote $\beta=y_F-\Delta(k)$ and $\beta'=y_F-\Delta(k')$. Suppose for a contradiction that $k\neq k'$. Then
\begin{equation*}
\card{\wcH_n^k\cap[0,p_n)}\ge \card{\langle\gcd(\beta_n,p_n)\rangle_{p_n}}\;\text{ and }\;
\card{\wcH_n^{k'}\cap[0,p_n)}\ge \card{\langle\gcd(\beta_n',p_n)\rangle_{p_n}},
\end{equation*}
where $\langle s\rangle{p_n}$ denotes the subgroup generated by $s$ in $\Z/p_n\Z$,
so that
\begin{equation*}
\card{\wcH_n\cap[0,p_n)}^2\ge \card{\wcH_n^{k}\cap[0,p_n)}\cdot\card{\wcH_n^{k'}\cap[0,p_n)}
\ge
\card{\langle \left(\gcd(\beta_n,p_n),\gcd(\beta_n',p_n)\right)\rangle_{p_n\times p_n}},
\end{equation*}
where $\langle (s,t)\rangle{p_n\times p_n}$ denotes the subgroup generated by $(s,t)$ in $(\Z/p_n\Z)^2$. Hence
\begin{equation*}
\begin{split}
\card{\wcH_n\cap[0,p_n)}^2
&\ge
\lcm\left(\frac{p_n}{\gcd(\beta_n,p_n)},\frac{p_n}{\gcd(\beta_n',p_n)}\right)
=\frac{p_n}{\gcd(\beta_n,\beta_n',p_n)}\\
&\ge \frac{p_n}{\gcd(\beta_n-\beta_n',p_n)}
=
\frac{p_n}{\gcd(k-k',p_n)}.
\end{split}
\end{equation*}
But the last denominator is at most $2m$, so
\begin{equation*}
\frac{1}{2m\cdot p_n}\le \left(\frac{\card{\wcH_n\cap[0,p_n)}}{p_n}\right)^2
=\ddelta(\wcH_n)^2
\end{equation*}
in contradiction to the assumption.
\end{remark}
Here is an example of a Toeplitz sequence for which Remark~\ref{claim} applies and for which $\tau_n$ and~$\wq_n$, the smallest periods of $\cH_n$ and $\wcH_n$, are different.
\begin{example}\label{ex:GarciaHedlung}
Garcia and Hedlund \cite{Garcia1948} gave the first example of $0$-$1$ nonperiodic Toeplitz sequence. At each level $n$ of their construction there is exactly one hole in each interval of length $p_n$, so all holes are essential, and the centralizer is trivial because condition (*) is satisfied, see Remark~\ref{remark:condition(*)}.
Our example is a modification of this construction: \footnote{To be precise, the example from \cite{Garcia1948} is not really a Toeplitz sequence, because it is not periodic at position $0$. But its orbit closure is minimal and contains many Toeplitz sequences. Our modification takes this into account.}

Let $r_n:=\sum_{j=0}^{n-1}2^{2j}=\frac{2^{2n}-1}{3}$. Define a Toeplitz sequence in such a way that
\begin{equation*}
\cH_n=2^{2n}\Z-r_n.
\end{equation*}
Observe that $\cH_{n}=2^{2(n-1)}(4\Z-1)-r_{n-1}\subseteq\cH_{n-1}$, in particular $\bigcap_{n\ge 1}\cH_n=\emptyset$.
$\cH_{n-1}\setminus\cH_{n}$ is the disjoint union of the residue classes
$2^{2(n-1)}(4\Z-k)-r_{n-1}$, $k\in\{0,2,3\}$, and the positions in each of these residue classes should be filled alternatingly with $0$ and $1$. Then all these positions have minimal period $2^{2n+1}$, and  $(p_n)_{n\ge 1}=(2^{2n+1})_{n\ge 1}$
is a period structure for the resulting Toeplitz sequence. 

If $2^{2N}t-r_N\in\cH_N$ and $n>N$, then
\begin{equation*}
\begin{split}
(2^{2N}t-r_N+p_N\Z)\cap\cH_n
&=
\left((2^{2N}t+(r_n-r_N)+p_N\Z)\cap 2^{2n}\Z\right)-r_n\\
&=
2^{2N}\left((t+\frac{2^{2(n-N)}-1}{3}+2\Z)\cap 2^{2(n-N)}\Z\right)-r_n
\end{split}
\end{equation*}
is non-empty if and only if $t$ is odd. Hence $\wcH_N=2^{2N}(2\Z+1)-r_N$ and $\wq_N=2^{2N+1}=p_N=2\tau_N$.
Notice that each interval of length $p_n$ contains exactly two holes and the distance between them is~$\frac{p_n}{2}$. But $p_{n+1}=4p_n$, so condition (*) is not satisfied. Nevertheless the centralizer is trivial by Theorem~\ref{theo:basic-case} and Remark~\ref{claim}.
\end{example}
\subsection{Additional arithmetic structure (motivated by the $\cB$-free case)}
\label{subsec:additional-structure}

Throughout this subsection $F$ is again an automorphism of $(X_\eta,\sigma)$ and $\pi(F(x))=\pi(x)+y_F$ for $x\in X_\eta$.
We start with a particularly simple situation based on the following (very strong) trivial intersection property: There are $A_n\subseteq\N$, $n\in\N$ such that
\begin{equation}\label{eq:trivial-intersection-new}\tag{TI}
\bigcap_{n\in\N}\gen{A_n}=\{0\}\text{ and }\wcH_n\subseteq\cM_{A_n},\; n\in\N.
\end{equation}
We will check this property for some nontrivial $\cB$-free examples, see Examples~\ref{ex:power:identity}, \ref{ex:not_all_holes_new} and~\ref{ex:holes} and also Subsection \ref{subsec:non-trivial}.
\begin{proposition}\label{prop:invariance-new}
Suppose \eqref{eq:trivial-intersection-new} is satisfied.
\begin{compactenum}[a)]
\item  Then the disjointness condition \eqref{D} holds, and there exists a unique $k\in\Z$ such that
$\partial W+y_F\subseteq \partial W+\Delta(k)$.
\item $\wq_n\mid(y_F)_n-k$, where $\wq_n$ is the minimal period of $\wcH_n$. In particular, if infinitely many $\wcH_n$ have minimal period $p_n$ then $y_F=\Delta(k)$.
\end{compactenum}
\end{proposition}

\begin{proof}a)\;
If there is some $y\in\partial W\cap(\partial W-\Delta(k))$, then
$U_{N}(\Delta(y_N))\cap\partial W=U_{N}(y)\cap\partial W\neq\emptyset$ and
$U_{N}(\Delta(k+y_N))\cap\partial W=U_{N}(\Delta(k)+y)\cap\partial W\neq\emptyset$,
so that $k=(k+y_N)-y_N\in\wcH_N-\wcH_N\subseteq\cM_{A_N}-\cM_{A_N}\subseteq\gen{\cM_{A_N}}=\gen{A_{N}}$ for all $N>0$ by Lemma~\ref{lemma:newG-2}a). Hence $k=0$ in view of \eqref{eq:trivial-intersection-new}. If there are $y_1,y_2\in \partial W$ and $k_1,k_2\in\Z$ such that $y_i+y_F\in\partial W-\Delta(k_i)$ ($i=1,2$), then
$U_N(\Delta(k_i+(y_i+y_F)_N))\cap\partial W\neq\emptyset$ ($i=1,2$), so that
$k_2-k_1\in \gen{A_N}$ for all $N>0$ as before. Hence $k_2=k_1$
because of \eqref{eq:trivial-intersection-new}.\\
b)\;This is Lemma~\ref{lemma:basic-new}b).
\end{proof}

Together with Theorem~\ref{theo:basic-case}, this proposition yields:
\begin{corollary}\label{coro:simple-case}
Suppose that \eqref{eq:trivial-intersection-new} is satisfied.
If $M:=\liminf_{n\to\infty}p_n/\wq_n<\infty$, then \[\Aut_\sigma(X_\eta)=\langle \sigma\rangle\oplus\Tor,\] where $\Tor$ denotes the torsion group of $\Aut_\sigma(X_\eta)$. It is a cyclic group  (possibly trivial), whose order divides $M$. In particular, if infinitely many $\wcH_n$ have minimal period $p_n$, then the centralizer of $(X_\eta,\sigma)$ is trivial.
\end{corollary}

If \eqref{eq:trivial-intersection-new} is not satisfied, as it is the case for more complex $\cB$-free examples, we need additional tools to verify assumption~\eqref{eq:unique_k} of Theorem~\ref{theo:basic-case}. The weak double disjointness condition~\eqref{DD'}
turns out to be instrumental along this way, hence its verification under mild arithmetic assumptions in Proposition~\ref{prop:local-invariance-new-2} will be an important step.

We continue with some arithmetic preparations.
For the sake of brevity we sometimes write $u\vee v$ instead of $\lcm(u,v)$.
The following notation will be used repeatedly for positive integers $a$ and $k$:
\begin{equation}
a^{\div k}:=\frac{a}{\gcd(a,k)}=\frac{a\vee k}{k}.
\end{equation}
For $A\subseteq \N$ and $k\in\N$ denote
\begin{equation}
A^{\div k}
:=
\left\{a^{\div k}: a\in A\right\},
\end{equation}
\begin{equation}\label{eq:intersections-3}
A^{\perp k}:=\left\{a\in A:\gcd\left(a,k\right)=1\right\},
\end{equation}
and
\begin{equation}
A^{prim}:=\{a\in A:\ a'\mid a\Rightarrow a'=a \text{\;\; for all }a'\in A\}.
\end{equation}
If $A=A^{prim}$ then $A$ is called \emph{primitive}.
\begin{remark}\label{remark:notations}
Let $r,\ell,s,m\in\Z$ and assume that $\gcd(m,\ell)\mid s-r$. Then
\begin{equation}\label{eq:intersections-1}
(r+\ell\Z)\cap(s+m\Z)
=
x+(\ell\vee m)\Z
=
g\cdot\left(\tilde{r}+\tilde{\ell}\Z\right),
\end{equation}
where $x\in\{0,\dots,(\ell\vee m)-1\}$ is defined uniquely by the first identity,  $g=\gcd(x,\ell\vee m)$, $\tilde{r}=\frac{x}{g}$, and $\tilde{\ell}=\frac{\ell\vee m}{g}$. Observe that $\gcd(\tilde{r},\tilde{\ell})=1$.
This formula will be applied in several settings, so that one should keep in mind that $g,\tilde{r}$, and $\tilde{\ell}$ depend on $r,\ell,s$, and $m$.
For later use observe also that
\begin{equation}\label{eq:g}
g=\gcd(r,\ell)\vee\gcd(s,m).
\end{equation}
Here is the proof: As $x-r\in\ell\Z$ and $x-s\in m\Z$, we have
$\gcd(x,\ell)=\gcd(r,\ell)$ and
$\gcd(x,m)=\gcd(s,m)$. Hence
\begin{equation*}
g=\gcd(x,\ell\vee m)=\gcd(x,\ell)\vee\gcd(x,m)=\gcd(r,\ell)\vee\gcd(s,m).
\end{equation*}

We list some further consequences:
\begin{equation}\label{eq:intersections-2}
\lcm(A^{\div g})=\lcm(A)^{\div g},\quad
g\cdot\cM_{A^{\div g}}=\cM_A\cap g\Z,\text{\; and \;}
\cM_{A^{\div g}}\subseteq\cM_{A^{\div\ell\vee m}},
\end{equation}
where we used $\gcd(a,g)\mid\gcd(a,\ell\vee m)$ for the last inclusion.
For each subset $Z\subseteq\Z$
holds
\begin{equation}\label{eq:intersections-4}
\left(\tilde{r}+\tilde{\ell}Z\right)\cap \cM_{A}
=
\left(\tilde{r}+\tilde{\ell}Z\right)\cap \cM_{A^{\perp{\tilde{\ell}}}},
\end{equation}
because $\gcd(\tilde{r},\tilde{\ell})=1$.\footnote{Indeed, if $a\in A$, $z\in Z$ and $x=\tilde{r}+z\tilde{\ell}\in a\Z$, then $\gcd(a,\tilde{\ell})\mid\tilde{r}$, so that
$\gcd(a,\tilde{\ell})\mid\gcd(\tilde{r},\tilde{\ell})=1$, i.e.~$a\in A^{\perp{\tilde{\ell}}}$.} Combining \eqref{eq:intersections-1}, \eqref{eq:intersections-2} and \eqref{eq:intersections-4} yields
\begin{equation}\label{eq:intersections-5}
(r+\ell\Z)\cap(s+a\Z)\cap\cM_A
=
g\cdot\left((\tilde{r}+\tilde{\ell}\Z)\cap\cM_{(A^{\div g})^{\perp{\tilde{\ell}}}}\right).
\end{equation}
\end{remark}
Given a set $A\subseteq \Z$ we denote by $\ddelta(A):=\lim_{N\to\infty}\frac{1}{\log N}\sum_{k=1}^N\frac{1}{k}1_A(k)$ the logarithmic density
 of $A$ (provided the limit exists).

\begin{lemma}\label{lemma:intersections-density}
Let $r,\ell,s,m\in\Z$ and assume that $\gcd(m,\ell)\mid s-r$.
Recall that $\tilde{\ell}=\frac{\ell\vee m}{g}$.
Then
\begin{equation}\label{eq:intersections-density-1}
\ddelta\left((r+\ell\Z)\cap(s+m\Z)\cap\cM_A\right)
=
\frac{1}{\ell\vee m}\cdot\ddelta\left(\cM_{(A^{\div g})^{\perp{\tilde{\ell}}}}\right)
\leqslant
\frac{1}{\ell\vee m}\cdot\ddelta\left(\cM_{A^{\ell\vee m}}\right),
\end{equation}
\begin{equation}\label{eq:intersections-density-2}
\ddelta\left((r+\ell\Z)\cap(s+m\Z)\setminus\cM_A\right)
=
\frac{1}{\ell\vee m}\cdot\left(1-\ddelta\left(\cM_{(A^{\div g})^{\perp{\tilde{\ell}}}}\right)\right).
\end{equation}
\end{lemma}
\begin{proof}
As in \cite[Lemma~1.17]{Hall1996} we have
\begin{equation*}
\ddelta\left((\tilde{r}+\tilde{\ell}\Z)\cap\cM_{(A^{\div g})^{\perp{\tilde{\ell}}}}\right)
=\frac{1}{\tilde{\ell}}\cdot\ddelta\left(\cM_{(A^{\div g})^{\perp{\tilde{\ell}}}}\right).
\end{equation*}
As $g\cdot\tilde{\ell}=\ell\vee m$, this together with (\ref{eq:intersections-1}) proves the identity in \eqref{eq:intersections-density-1}. As $\ddelta\left((r+\ell\Z)\cap(s+m\Z)\right)=\frac{1}{\ell\vee m}$, \eqref{eq:intersections-density-2} follows at once. For the inequality in \eqref{eq:intersections-density-1} observe that
$\cM_{(A^{\div g})^{\perp{\tilde{\ell}}}}\subseteq\cM_{A^{\div g}}\subseteq\cM_{A^{\div\ell\vee m}}$ by~\eqref{eq:intersections-2}.
\end{proof}
\begin{lemma}\label{finite_mult}
Assume that $(r+\ell\Z)\cap(s+m\Z)\cap[N,\infty)\subseteq \cM_C$ for some $r,\ell,s,m\in\Z$ satisfying $\gcd(m,\ell)\mid s-r$, some $N\in \N$  and a finite set $C\subset \N$. Then $c\mid\gcd(r,\ell)\vee\gcd(s,m)$ for some $c\in C$.
\end{lemma}

\begin{proof}
Let $A:=(C^{\div g})^{\perp\tilde \ell}$.
In view of~\eqref{eq:intersections-5}, our assumption implies
$(\tilde{r}+\tilde{\ell}\Z)\cap[N,\infty)\subseteq\cM_{A}=\cM_{A^{prim}}$.
As $A^{prim}$ is taut,
Proposition~4.31 in \cite{BKKL2015} shows that there is $a\in A^{prim}$ such that $a\mid\gcd(\tilde{r},\tilde{\ell})=1$, i.e. $1\in A$. Hence there is $c\in C$ such that $c\mid g=\gcd(r,\ell)\vee\gcd(s,m)$.
\end{proof}
We will need a more detailed arithmetic characterization of the inclusion from Lemma~\ref{finite_mult} when $C$ is a singleton and $s=0$.

\begin{lemma}\label{arith_char_incl}
Let $r,\ell,a,c\in\Z$ satisfying $\gcd(a,\ell)\mid r$. Then the following conditions are equivalent:
\begin{compactenum}[a)]
\item $(r+\ell\Z)\cap a\Z\subseteq c\Z$,
\item $c\mid \gcd(r,\ell)\vee a$,
\item $c\mid \ell\vee a$ and $\gcd(c,\ell)\mid r$.
\end{compactenum}
\end{lemma}
\begin{proof}
By Lemma \ref{finite_mult}, a) implies b). Conversely, a) follows from b), because
$(r+\ell\Z)\cap a\Z\subseteq(\gcd(r,\ell)\vee a)\Z$.

Suppose that a) and b) hold. By b), we have $c\mid \gcd(r,\ell)\vee a\mid\ell\vee a$. Since $\gcd(a,\ell)\mid r$, we have $(r+\ell\Z)\cap a\Z\neq\emptyset$. So by a), we get $r\in c\Z+\ell\Z=\gcd(c,\ell)\Z$. Hence c) holds.

Finally suppose that c) holds. Then
\begin{equation*}
\gcd(c,\gcd(r,\ell)\vee a)=\gcd(c,r,\ell)\vee\gcd(c,a)=\gcd(c,\ell)\vee\gcd(c,a)=\gcd(c,\ell\vee a)=c,
\end{equation*}
and b) follows at once.
\end{proof}
We will need to know the smallest periods of the difference of the sets of multiples.

\begin{lemma}\label{lem:describe_period}
Assume that $A$ and $C$ are finite subsets of $\N$ and that the set $A$ is primitive.
\begin{compactenum}[a)]
\item If $A=\{a\}$ and $a\not\in\cM_C$, then $a\cdot\lcm((C^{\div a})^{prim})$ is the minimal period of $a\Z\setminus \cM_C$.
\item If the set $C^{\div a}=\{\frac{c}{\gcd(a,c)}:c\in C\}\subset \N\setminus\{1\}$ is primitive for every $a\in A$, then $\lcm(A \cup C)$ is the minimal period of $\cM_{A}\setminus \cM_C$.
\end{compactenum}
\end{lemma}

\begin{proof}
a)\;
Let $a\in\N\setminus\cM_C$. Then $\emptyset\neq a\Z\setminus \cM_C=a\cdot(\Z\setminus \cM_{C^{\div a}})$, and the claim follows from the observation that
$\lcm((C^{\div a})^{prim})$ is the minimal period of $\cM_{C^{\div a}}$, see \cite[Lemma~5.1b)]{KKL2016}.\\
b)\; Let $T$ be  the minimal period of $\cM_A\setminus \cM_C$. Clearly, $T|\lcm(A\cup C)$. Now let $a\in A$. Since $1\not\in C^{\div a}$,
$a\in \cM_A\setminus\cM_C$, so $a+T\Z\subset \cM_A$. By Lemma~\ref{finite_mult} and primitivity of $A$ we get $a|T$, so that $T+(a\Z\setminus \cM_C)\subset a\Z\setminus \cM_C$ for every $a\in A$
and therefore $a\cdot\lcm((C^{\div a})^{prim})\mid T$ in view of part a). Hence
$\lcm(\{a\}\cup C)=a\lcm(C^{\div a})=a\lcm((C^{\div a})^{prim})\mid T$ for every $a\in A$, so that $\lcm(A\cup C)|T$.
\end{proof}
The following proposition is the ``multi-tool'' of this section:
\begin{proposition}\label{prop:central-new}
Let $A_n$ and $S_n$ be sets of positive integers. Let $b\in\Z$,
$n\geqslant N\geqslant0$,
$a\in A_n$ and $r\in\Z$ be such that
\begin{equation}\label{eq:key-ass-00-new}
(r+p_N\Z)\cap (a\Z)\setminus\cM_{S_n}\neq\emptyset
\end{equation}
and
\begin{equation}\label{eq:key-ass-0-new}
\left((r+p_N\Z)\cap (a\Z)\setminus\cM_{S_n}\right)+{b}\subseteq\cM_{A_n}.
\end{equation}
Assume that there is a subset $E_{n,a}$ of $A_n\setminus\{a\}$ for which
\begin{equation}\label{eq:to-check-1-new}
\sum_{a'\in E_{n,a}}\frac{1}{\varphi(a'^{\div a})}<\frac{1}{p_N},
\end{equation}
where $\varphi$ is Euler's totient function.
Then there is $a'\in A_n\setminus E_{n,a}$ such that
\begin{equation}\label{eq:pos-dens-a'-new}
{b}\in\gcd(a',\gcd(r,p_N)\vee a)\Z \text{ and }
\emptyset\neq\left((r+p_N\Z)\cap (a\Z)\setminus\cM_{S_n}\right)\cap\left(-{b}+a'\Z\right).
\end{equation}
Applying this to ${b}=\beta_n$\, for some $\beta=\Delta(k)\in\Delta(\Z)$ with
$|k|<\gcd(a',a)$ for all $a'\in A_n\setminus E_{n,a}$, implies $\beta=0$.
\end{proposition}
\begin{proof}
In view of assumption \eqref{eq:key-ass-0-new} we have
\begin{equation}\label{eq:48-new}
\begin{split}
\left((r+p_N\Z)\cap (a\Z)\setminus\cM_{S_n}\right)\subseteq
\left(-{b}+\cM_{E_{n,a}}\right)\cup\left(-{b}+ \cM_{D_{n,a}}\right),
\end{split}
\end{equation}
where $D_{n,a}=A_n\setminus E_{n,a}$. We prove below that \eqref{eq:to-check-1-new} implies
\begin{equation}\label{eq:key-ass-1-new}
\left((r+p_N\Z)\cap(a\Z)\setminus\cM_{S_n}\right)+{b}\not\subseteq\cM_{E_{n,a}}.
\end{equation}
Hence and by \eqref{eq:48-new} there is $a'\in D_{n,a}$ such that
\eqref{eq:pos-dens-a'-new} holds.
Hence ${b}=0$ or $|{b}|\geqslant \gcd(a',\gcd(r,p_N)\vee a)\geqslant\gcd(a',a)$.\\
It remains to show that \eqref{eq:to-check-1-new} implies \eqref{eq:key-ass-1-new}.

Consider any
$a\in A_n$ for which
$
\left((r+p_N\Z)\cap(a\Z)\setminus\cM_{S_n}\right)\neq\emptyset
$.
Then $\gcd(a,p_N)\mid r$, and
Lemma~\ref{lemma:intersections-density} implies
\begin{equation}\label{eq:example-1-new}
\begin{split}
0<\ddelta\left((r+p_N\Z)\cap(a\Z)\setminus\cM_{S_n}\right)
&=
\frac1{p_N\vee a}\left(1-\ddelta\left(\cM_{(S_n^{\div g_n})^{\perp(p_N^{\div g_n})}}\right)\right)
,
\end{split}
\end{equation}
where
\begin{equation*}
g_n:=\gcd(r,p_N)\vee a\quad\text{and}\quad
p_N^{\div g_n}=\frac{p_N\vee g_n}{g_n}=\frac{p_N\vee a}{g_n}
\mid p_N^{\div a}.
\end{equation*}
Suppose for a contradiction that there is inclusion in \eqref{eq:key-ass-1-new}.
Then
\begin{equation*}
(r+p_N\Z)\cap(a\Z)\setminus\cM_{S_n}
\subseteq\bigcup_{a'\in E_{n,a}}(-{b}+a'\Z),
\end{equation*}
so that
\begin{equation}\label{eq:example-1a-new}
\begin{split}
&\ddelta\left((r+p_N\Z)\cap(a\Z)\setminus\cM_{S_n}\right)\\
&\leqslant
\sum_{a'\in E_{n,a}}\ddelta\left((r+p_N\Z)\cap(-{b}+a'\Z)\cap(a\Z)\setminus\cM_{S_n}\right)\\
&\leqslant
\sum_{a'\in E_{n,a}}
\ddelta\left((g_n\Z)\cap(-{b}+a'\Z)\setminus\cM_{S_n}\right)
\\
&=
\sum_{a'\in E_{n,a},\ \gcd(g_n,a')\mid{b}}
\ddelta\left((g_n\Z)\cap(-{b}+a'\Z)\setminus\cM_{S_n}\right)
\\
&=
\sum_{a'\in E_{n,a},\ \gcd(g_n,a')\mid{b}}
\frac{1}{g_n\vee a'}\left(1-
\ddelta\left(\cM_{(S_n^{\div (g_n\vee\gcd({b},a'))})^{\perp\frac{g_n\vee a'}{g_n\vee\gcd({b},a')}}}\right)\right).
\end{split}
\end{equation}
Notice that $\gcd(g_n,a')\mid{b}$ implies
\begin{equation}\label{eq:simplification}
\frac{g_n\vee a'}{g_n\vee\gcd({b},a')}
=
\frac{(g_n\vee a')\cdot\gcd(g_n,a')}{g_n\cdot\gcd({b},a')}=\frac{a'}{\gcd({b},a')}=a'^{\div b}.
\end{equation}
As
\begin{equation}\label{eq:quotient-new}
\frac{p_N\vee a}{a'\vee g_n}
\leqslant
\frac{p_N\vee a}{a'\vee a}
=
\frac{p_N^{\div a}}{a'^{\div a}},
\end{equation}
\eqref{eq:example-1-new}, \eqref{eq:example-1a-new} and \eqref{eq:simplification} together yield
\begin{equation}\label{eq:example-1b-new}
\begin{split}
1-\ddelta\left(\cM_{(S_n^{\div g_n})^{\perp(p_N^{\div g_n})}}\right)
&\leqslant
\sum_{a'\in E_{n,a}}
\frac{p_N^{\div a}}{a'^{\div a}}\left(1-
\ddelta\left(\cM_{(S_n^{\div(g_n\vee\gcd({b},a'))})^{\perp{a'^{\div b}}}}\right)\right).
\end{split}
\end{equation}
Denote
\begin{equation*}
\begin{split}
R_n(a')
:=&
\left\{b^{\div g_n}: b\in S_n,\ b^{\div g_n}\perp p_N^{\div g_n},\
b^{\div g_n}\not\perp a'^{\div g_n}\right\}.
\end{split}
\end{equation*}
Then $R_n(a')\subseteq S_n^{\div g_n}$ trivially, and
we claim that
\begin{equation*}
(S_n^{\div g_n})^{\perp(p_N^{\div g_n})}\setminus R_n(a')
\subseteq
\cM_{(S_n^{\div(g_n\vee\gcd({b},a'))})^{\perp{a'^{\div b}}}}.
\end{equation*}
 Indeed, each $b^{\div g_n}\in S_n^{\div g_n}\setminus R_n(a')$, which is also coprime to $p_N^{\div g_n}$, is coprime to $a'^{\div g_n}$ and, a fortiori, to
$a'^{\div b}$ because $\gcd(g_n,a')\mid b$.
Moreover, $b^{\div g_n}=\frac{b\vee g_n}{g_n}$ is a multiple of
$\frac{b\vee g_n\vee\gcd({b},a')}{g_n\vee\gcd({b},a')}$, so that the latter is also coprime to
${a'^{\div b}}$.
 Therefore,
\begin{equation*}
(S_n^{\div g_n})^{\perp(p_N^{\div g_n})}
\subseteq
R_n(a')\cup
\cM_{(S_n^{\div(g_n\vee\gcd({b},a'))})^{\perp{a'^{\div b}}}},
\end{equation*}
so that Behrend's inequality (see \cite[Thm.~0.12]{Hall1996} for a reference)
yields
\begin{equation*}
\begin{split}
1-\ddelta\left(\cM_{(S_n^{\div g_n})^{\perp(p_N^{\div g_n})}}\right)
&\geqslant
1-\ddelta\left(\cM_{R_n(a')\cup
\cM_{\left(S_n^{\div(g_n\vee\gcd({b},a'))}\right)^{\perp{a'^{\div b}}}}}\right)\\
&=
1-\ddelta\left(\cM_{R_n(a')\cup
{(S_n^{\div(g_n\vee\gcd({b},a'))})^{\perp{a'^{\div b}}}}}\right)\\
&\geqslant
\left(1-\ddelta\left(\cM_{R_n(a')}\right)\right)\cdot
\left(1-
\ddelta\left(\cM_{(S_n^{\div(g_n\vee\gcd({b},a'))})^{\perp{a'^{\div b}}}}\right)\right).
\end{split}
\end{equation*}
Therefore \eqref{eq:example-1b-new} leads to
\begin{equation*}
1-\ddelta\left(\cM_{(S_n^{\div g_n})^{\perp(p_N^{\div g_n})}}\right)
\leqslant
\sum_{a'\in E_{n,a}}
\frac{{p_N}^{\div a}}{a'^{\div a}}\cdot
\frac{1-\ddelta\left(\cM_{(S_n^{\div g_n})^{\perp(p_N^{\div g_n})}}\right)}{1-\ddelta\left(\cM_{R_n(a')}\right)}.
\\
\end{equation*}
As $1-\ddelta\left(\cM_{(S_n^{\div g_n})^{\perp(p_N^{\div g_n})}}\right)>0$ in view of \eqref{eq:example-1-new}, we can divide the last inequality by this expression, so that
\begin{equation}\label{eq:contradiction-B1-new}
\begin{split}
1
&\leqslant
\sum_{a'\in E_{n,a}}
\frac{p_N^{\div a}}{a'^{\div a}}\cdot
\frac{1}{1-\ddelta\left(\cM_{R_n(a')}\right)}
\leqslant
\sum_{a'\in E_{n,a}}
\frac{p_N^{\div a}}{a'^{\div a}}\cdot
\frac{1}{1-\ddelta\left(\cM_{\Spec(a'^{\div a})}\right)}\\
&=
\sum_{a'\in E_{n,a}}
\frac{{p_N}^{\div a}}{a'^{\div a}}\cdot
\prod_{p\mid a'^{\div a}}\frac{1}{1-\frac{1}{p}}
\leqslant
{p_N}\cdot \sum_{a'\in E_{n,a}}\frac{1}{\varphi(a'^{\div a})},
\end{split}
\end{equation}
where we used the fact that $R_n(a')\subseteq\cM_{\Spec(a'^{\div g_n})}\subseteq\cM_{\Spec(a'^{\div a})}$. But the last estimate contradicts assumption
\eqref{eq:to-check-1-new}.
\end{proof}

From now on we assume that the sets $\wcH_n$ have some particular \emph{arithmetic structure}: There is a primitive set $A_n$ of positive integers
such that for each $a_n\in A_n$ there is a set $S_n=S_n(a_n)$ of positive integers satisfying
\begin{equation}\label{eq:additional-structure}\tag{AS}
\wcH_n=\bigcup_{a_n\in A_n}a_n\Z\setminus\cM_{S_n(a_n)}
\text{\quad and\quad}a_n\Z\setminus\cM_{S_n(a_n)}\neq\emptyset\;(a_n\in A_n).
\end{equation}
Observe that $\min A_n\to\infty$ as $n\to\infty$, because $A_n\subseteq\wcH_n\subseteq\cH_n$ and $\min\cH_n\to\infty$.

In the remaining part of this section we prove the weak double disjointness condition \eqref{DD'} under the
arithmetic structure  assumption~\eqref{eq:additional-structure}, which allows to apply Proposition~\ref{prop:local-invariance-new} in this situation. Later, in Theorem~\ref{theo:ess-holes-arithmetic}, we verify \eqref{eq:additional-structure} in the $\cB$-free setting.

\begin{proposition}\label{prop:sufficient-for-(D')}
Assume condition \eqref{eq:additional-structure}.
If
\begin{equation}\label{eq:theo2-ass-2-new-a}
\lim_{n\to\infty}
\sum_{a'\in A_n}
\frac{1}{\varphi({a'})}
=0,
\end{equation}
then the weak disjointness condition \eqref{D'} is satisfied.

Moreover, if the automorphism $F$ of~$(X_\eta,\sigma)$ is described by a block code $\{0,1\}^{[-m:m]}\to\{0,1\}$, then the set $K$ is contained in $[-m,m]$, $\inn_{\partial W}(V_{k_i})\cap\inn_{\partial W}(V_{k_j})=\emptyset$ for any different $k_i,k_j\in K$, and
$\partial W=\bigcup_{k\in K}V_k'$, where $V_k':=\overline{\inn_{\partial W}(V_k)}$.
\end{proposition}

\begin{proof} Abbreviate $y=y_F$. Suppose for a contradiction that \eqref{D'} does not hold, equivalently that \eqref{Seh'} does not hold. Then there are $k\in\Z\setminus\{0\}$ and an arithmetic progression $r+p_N\Z$ such that for all $n\ge N$
\begin{equation*}
\emptyset\neq(r+p_N\Z)\cap\wcH_n\subseteq \wcH_n-k.
\end{equation*}
Let $n\ge N$. In view of property \eqref{eq:additional-structure},
there is $a\in A_n$
such that
\begin{equation}\label{eq:inclusions-3a-new}
\emptyset\neq(r+p_N\Z)\cap a\Z\setminus\cM_{S_n}+k\subseteq \cM_{A_n}.
\end{equation}
Let $E_{n,a}=\{a'\in A_n: \gcd(a',a)\leqslant|k|\}$.
Since
\begin{equation*}
\varphi(a'^{\div a})=a'^{\div a}\prod_{p\mid a'^{\div a}}\left(1-\frac{1}{p}\right)\geqslant\frac{1}{|k|}a'\prod_{p\mid a'}\left(1-\frac{1}{p}\right)=\frac{1}{|k|}\varphi(a')
\end{equation*}
for any $a'\in E_{n,a}$,
assumption \eqref{eq:theo2-ass-2-new-a} above implies assumption \eqref{eq:to-check-1-new} of Proposition~\ref{prop:central-new}. So this
proposition applies to the inclusion in~\eqref{eq:inclusions-3a-new},
and there is $a'\in A_n\setminus E_{n,a}$
such that $\gcd(a',a)\mid k$.
As $|k|<\gcd(a',a)$ for all $a'\in A_n\setminus E_{n,a}$,
this contradicts the assumption $k\in\Z\setminus\{0\}$.

The remaining conclusions follow from Corollary~\ref{coro:new-1-new}.
\end{proof}

\begin{remark}\label{rem:regular1}
Any $\cB$-free Toeplitz subshift satisfying \eqref{eq:theo2-ass-2-new-a} is regular, because $\wcH_n\subseteq\cM_{\cA_{S_n}^{\infty}}$ and $d(\cM_{\cA_{S_n}^{\infty}})\leq\sum_{a\in\cA_{S_n}^{\infty}}\frac{1}{a}\leq\sum_{a\in\cA_{S_n}^{\infty}}\frac{1}{\varphi(a)}$.
\end{remark}

\begin{proposition}\label{prop:local-invariance-new-2}
Assume condition \eqref{eq:additional-structure}.
If
\begin{equation}\label{eq:theo2-ass-2-new}
\lim_{n\to\infty}
\sum_{a'\in A_n\setminus\{a\}}
\frac{1}{\varphi({a'}^{\div a})}
=0\quad\text{for all choices of }a\in A_n\; (\text{where }{a'}^{\div a}={a'}/{\gcd(a',a)}),
\end{equation}
then the weak double disjointness condition \eqref{DD'} -- and a fortiori condition \eqref{D'} -- is satisfied.

Moreover, the conclusions of Proposition~\ref{prop:sufficient-for-(D')} can be complemented by
$V_k'+(y_F-\Delta(k))\Z\subseteq V_k'$ for all $k\in K$.
\end{proposition}

\begin{proof} Abbreviate $y=y_F$. Suppose for a contradiction that \eqref{DD'} does not hold, equivalently that \eqref{DSeh'} does not hold. Then there are $k\in\Z\setminus\{0\}$, $\beta\in G$, an arithmetic progression $r+p_N\Z$ such that for all $n\ge N$
\begin{equation*}
\emptyset\neq(r+p_N\Z)\cap\wcH_n\subseteq (\wcH_n-\beta_n)\cap(\wcH_n-2\beta_n-k).
\end{equation*}
Let $n\ge N$. In view of property \eqref{eq:additional-structure},
there is $a\in A_n$
such that
\begin{equation}\label{eq:inclusions-3-new}
\emptyset\neq(r+p_N\Z)\cap a\Z\setminus\cM_{S_n}+\beta_n\subseteq \cM_{A_n}\text{ and }
\emptyset\neq(r+p_N\Z)\cap a\Z\setminus\cM_{S_n}+2\beta_n+k\subseteq \cM_{A_n}.
\end{equation}
Let $E_{n,a}= A_n\setminus\{a\}$.
In view of assumption \eqref{eq:theo2-ass-2-new},
Proposition~\ref{prop:central-new} applies to both inclusions in~\eqref{eq:inclusions-3-new}, and as $A_n\setminus E_{n,a}=\{a\}$, we can conclude that $a\mid\beta_n$ and $a\mid 2\beta_n+k$, so that $a\mid k$.
As $a\in A_n$ and $\min A_n\to\infty$, this contradicts the assumption
$k\in\Z\setminus\{0\}$.

The final conclusion follows from Proposition~\ref{prop:local-invariance-new}.
\end{proof}
\begin{theorem}\label{theo:pre-B-free}
Assume condition \eqref{eq:additional-structure} and let the automorphism $F$ of~$(X_\eta,\sigma)$ be described by a block code $\{0,1\}^{[-m:m]}\to\{0,1\}$.
Under assumption~\eqref{eq:theo2-ass-2-new} of Proposition~\ref{prop:local-invariance-new-2} holds:
\begin{compactenum}[a)]
\item For each $n>0$ and each $k\in K$ there exists $a\in A_n$ such that
\begin{equation*}
a\mid (y_F)_{n}-k.
\end{equation*}
\item
For each $n>0$ and each $a\in A_n$ there exists some $k\in K$ such that
\begin{equation*}
a\mid (y_F)_{n}-k.
\end{equation*}
If $a>2m$, then this $k\in K$ is unique. Denote it by $\kappa_n(a)$.
\item Suppose $n$ is so large that $\min A_n>2m$, and denote by $\cG_n$ the graph with vertices $A_n$ and edges $(a,a')$ whenever $\gcd(a,a')>2m$. Then $\kappa_n(a)=\kappa_n(a')$ for any two $a,a'$ in the same connected component of $\cG_n$. In particular, $\card K=1$ if $\cG_n$ is connected.
\item If $\card K=1$, say $K=\{k\}$, then, for each $n$, $(y_F)_n-k$ is a multiple of the minimal period $\wq_n$ of $\wcH_n$.
\item If $\card K=1$ and $\wq_n=p_n$ for all $n$, then $(X_\eta,\sigma)$ has a trivial centralizer.
\end{compactenum}
\end{theorem}
\begin{proof}
a)\;
Let $k\in K$. Then $V_k'+(y_F-\Delta(k))\Z\subseteq V_k'\subseteq\partial W$ by Proposition~\ref{prop:local-invariance-new-2}, and because of Lemma~\ref{lemma:newG-2},
$h_n+(y_F-\Delta(k))_n\Z\subseteq\wcH_n\subseteq\cM_{A_n}$ for each $h\in V_k'$.
It follows that
there exists $a'\in A_n$ such that $a'\mid\gcd(h_{n},(y_F-\Delta(k))_{n})$.
\\
b)\;
Let $a\in A_n$. Notice that $a\not\in\cM_{S_n(a)}$. Otherwise, $a\Z\subseteq\cM_{S_n(a)}$ which contradicts \eqref{eq:additional-structure}. So $a\in\wcH_n$ and, by Lemma~\ref{lemma:newG-2}, there exists some $h\in U_{n}(\Delta(a))\cap \partial W$. In particular, $h_{n}=a\in\wcH_n$.
Because of Corollary~\ref{coro:new-1-new},
there exists $k\in K$ such that $h\in V_k'$.
As in the proof of a) it follows that
there exists $a'\in A_n$ such that $a'\mid\gcd(h_{n},(y_F-\Delta(k))_{n})$. As $a'$ and $a=h_n$ belong to the same primitive set $A_n$, this implies $a=a'\mid (y_F-\Delta(k))_{n}$.

Suppose there is another $k'\in K$ such that $a\mid (y_F-\Delta(k'))_{n}$. Then
$a\mid k-k'$, so that $k=k'$ or $a\leqslant |k-k'|\leqslant 2m$.\\
c)\;It suffices to prove that $\kappa_n(a)=\kappa_n(a')$ for every edge $(a,a')$ of $\cG_n$, i.e.~whenever $\gcd(a,a')>2m$. But as b) implies
\begin{equation*}
\gcd(a,a')\mid\left((y_F)_{n}-\kappa_n(a)\right)-\left((y_F)_{n}-\kappa_n(a')\right)=\kappa_n(a')-\kappa_n(a),
\end{equation*}
it follows that $\kappa_n(a)=\kappa_n(a')$ or
$2m<\gcd(a,a')\leqslant|\kappa_n(a)-\kappa_n(a')|\leqslant2m$.\\
d)\;If $K=\{k\}$, then $\partial W=V_k'$, i.e. $\partial W+(y_F-\Delta(k))\subseteq\partial W$, and the claim follows from Lemma~\ref{lemma:basic-new}b).\\
e)\; It follows from d) that $y_F=\Delta(k)$ for some $k\in\Z$.
\end{proof}
\section{The $\cB$-free case}\label{sec:Bfree}
In this section we will apply our results for general Toeplitz subshifts from section \ref{sec:abstract_Toeplitz} to minimal $\cB$-free subshifts.

\subsection{Preparations}\label{sec:topological-preparations}
Let us start with the following notation and observations.
\begin{compactitem}[-]
\item For a finite subset $S\subset\cB$ define as in \cite{KKL2016}
\begin{equation*}
\ell_S:=\lcm(S)\quad\text{and}\quad
\cA_S:=\{\gcd(b,\ell_S):b\in\cB\}.
\end{equation*}
As $\cB$ is primitive, $S$ is a proper subset of $\cA_S$.
\item The set $\cB$ is taut, if
$\delta(\cM_{\cB\setminus\{b\}})<\delta(\cM_\cB)$ for each $b\in\cB$.

So a set is primitive, if removing any single point from it changes its set of multiples, and a set is taut, if removing any single point from it changes the logaritmic density of its set of multiples.
\item Let $S\subseteq S' \subset\cB$. From \cite[Eq.~(17)]{KKL2016} we recall that
\begin{equation}\label{eq:A_S-inclusions}
S\subseteq S'\subseteq\cA_{S'}\subseteq\cM_{\cA_S}\text{\quad so that\quad}
\cM_S\subseteq\cM_{S'}\subseteq\cM_{\cA_{S'}}\subseteq\cM_{\cA_S}.
\end{equation}
\item A finite set $S\subset\cB$ is \emph{saturated}, if $\cA_S\cap\cB=S$.
\item For a finite set $S\subset\cB$ define $S^\sat=\cA_S\cap\cB$.
\end{compactitem}
Then $S\subseteq S^\sat$, $S^\sat$ is finite, $\lcm(S^\sat)\mid\lcm(\cA_S)=\lcm(S)$, so that $\lcm(S^\sat)=\lcm(S)$,  and $\cA_{S^\sat}=\cA_S$, because
$\gcd(b,\lcm(S^\sat))=\gcd(b,\lcm(S))$ for each $b\in\cB$. In particular, $\cA_{S^\sat}\cap\cB=\cA_S\cap\cB=S^\sat$, so that $S^\sat$ is saturated.

Any filtration $S_1\subseteq S_2\subseteq\ldots $ of $\cB$ by finite sets yields a period structure $p_n=\lcm(S_n)$ for $X_{\eta}$. The definition of the group $G$ depends on the period structure, but $G$ is naturally isomorphic with the inverse limit $\lim\limits_{\leftarrow}\Z/\lcm(S)\Z$ of the inverse system of cyclic groups $\Z/\lcm(S)\Z$ indexed by the finite subsets $S\subset\cB$ ordered by the inclusion. Moreover, there is an injective group homomorphism $\lim\limits_{\leftarrow}\Z/\lcm(S)\Z\rightarrow \prod_{b\in\cB}\Z/b\Z$
 given by $(n_S)_{S\subset\cB}\mapsto (n_{\{b\}}+b\Z)_{b\in \cB}$. We can identify the group $G$ with the image of this homomorphism, which consists of the elements $h=(h_b)_{b\in \cB}\in \prod_{b\in\cB}\Z/b\Z$ satisfying $h_b=h_{b'}\mod \gcd(b,b')$ for any $b,b'\in\cB$. Under this identification, $\Delta:\Z\to\prod_{b\in\cB}\Z/b\Z$ is given by
$(\Delta(n))_b=n+ b\Z$ for $b\in\cB$ and $\overline{\Delta(\Z)}\cong G$.
  Given a sequence $(n_S)_{S\subset\cB}$ of integers belonging to the inverse limit (that is,  satisfying $n_S\equiv n_{S'}\mod \lcm(S)$ whenever $S\subseteq S'$), we denote by $\lim\Delta(n_S)$ the element $h\in \overline{\Delta(\Z)}\subseteq\prod_{b\in\cB}\Z/b\Z$ such that $h_b= n_S\mod b$ for $b\in S\subset \cB$.
\begin{remark}\label{remark:coding-function}
The coding function $\phi:G\to\{0,1\}^\Z$ defined in the introduction can be written as $\phi(y)=\mathbf{1}_{\Z\setminus\bigcup_{b\in\cB}(b\Z-y_b)}$ for any $y=(y_b)_{b\in\cB}\in\overline{\Delta(\Z)}\cong G$. It is injective. Indeed, for $y\in G$ denote $I_y:=\{s\in\Z:(\phi(y))_s=1\}=\{s\in\Z:y+\Delta(s)\in W\}$.
As $\overline{\inn(W)}=W\subseteq\overline{\Delta(\Z)}$ (see the introduction), we have $W=\overline{\inn(W)}\subseteq\overline{\{y+\Delta(s):s\in I_y\}}\subseteq W$ for each $y\in G$. Hence, if $\phi(y)=\phi(y')$, then $I_y=I_{y'}$ and
\begin{equation*}
W=\overline{\{y+\Delta(s):s\in I_y\}}
=
\overline{\{y'+\Delta(s):s\in I_{y'}\}}+(y-y')
=
W+(y-y').
\end{equation*}
But $W$ is aperiodic \cite[Prop.~5.1]{KKL2016}, so $y=y'$.
\end{remark}

\begin{lemma}[{\cite[Lemma 2.5]{KKL2016}}]\label{lemma2.5}
Let $U = U_S(\Delta(n))$ for some $S\subset\cB$ and $n\in\Z$.
\begin{compactenum}[a)]
\item If $n\in\mathcal{M}_S$ , then $U\cap W =\emptyset$.
\item If $U\cap W=\emptyset$, then $n+\lcm(S)\cdot\Z\subseteq\mathcal{M}_{\cB\cap \cA_S}$.
\item If $S$ is saturated, then $n\in\mathcal{M}_S$ iff $U\cap W = \emptyset$ iff $n + lcm(S)\cdot\Z\subseteq\mathcal{M}_S$.
\end{compactenum}
\end{lemma}
\begin{lemma}[{\cite[Lemma 3.1]{KKL2016}}]\label{lemma3.1}
\begin{compactenum}[a)]
\item For all $S\subset\cB$ and $n\in\Z$ we have: $U_S(\Delta(n))\subseteq W\Leftrightarrow n\in\cF_{\cA_S}$.
\item If $(S_k)_k$ is a filtration of $\cB$ by finite sets and $\lim_k \Delta(n_{S_k}) = h$, then $h\in \inn(W)$ if
and only if $n_{S_k}\in\cF_{\cA_{S_k}}$ for some $k$.
\end{compactenum}
\end{lemma}
\begin{lemma}[{\cite[Lemma 5.2]{KKL2016}}]\label{lemma5.2}
Assume that $S\subseteq S'$ are finite subsets of $\cB$, then $\cA_S = \{\gcd(a, \lcm(S))\ : a \in\cA_{S'} \}$.
\end{lemma}
\begin{proposition}[{\cite[Theorem B]{KKL2016}}]\label{prop:thmB}
The following are equivalent:
\begin{compactenum}[a)]
\item $W$ is topologically regular, i.e. $W = \overline{\inn(W)}$.
\item There are no $d\in\N$ and no infinite pairwise coprime set $\cA\subset\N\setminus\{1\}$ such that $d\cA\subset\cB$.
\item  $\eta=\phi(0)$ is a Toeplitz sequence different from $(\ldots, 0, 0, 0,\ldots)$.
\item $\{n\in\N \ : \forall_{S \subset\cB}\  \exists_{S'\subset\cB}\ :S \subseteq S' \ \text{ and } \ n\in\cA_{S'}\setminus S'\}=\emptyset$.
\end{compactenum}
\end{proposition}

\begin{lemma}
Each filtration $S_1\subset S_2\subset\dots\nearrow\cB$ has a sub-filtration of sets $S_{n_k}$ such that $S_{n_1}^\sat\subset S_{n_2}^\sat\subset\dots\nearrow\cB$ is a filtration.
\end{lemma}

\begin{proof}
Let $n_1=1$. If $n_1<n_2<\dots<n_k$ are chosen, let $n_{k+1}=\min\{j\in\N: S_{n_k}^\sat\subseteq S_j\}$.
\end{proof}
\noindent
This lemma allows in the sequel to assume that a filtration is saturated without loosing generality.
\subsection{Sets of holes}

Now we describe the set of holes and the set of essential holes.

\begin{proposition}\label{des_holes}
Let $S \subset \mathcal{B}$ be saturated and taut and $s\in\Z$. Then
\begin{compactenum}[a)]
\item $s\in\cM_S$ $\Leftrightarrow$ $s+\ell_S\Z\subseteq\cM_S$ $\Leftrightarrow$ $s+\ell_S\Z\subseteq\cM_\cB$,
\item $s\in\cF_{\cA_S}$ $\Leftrightarrow$ $s+\ell_S\Z\subseteq\cF_{\cA_S}$  $\Leftrightarrow$ $s+\ell_S\Z\subseteq\cF_{\cB}$.
\end{compactenum}
In particular,
$s \in \Z$ is not
$\ell_S$-periodic if and only if $s \in \mathcal{M}_{\cA_{S}} \setminus \mathcal{M}_{S}$. Hence $\mathcal{M}_{\cA_{S_n}} \setminus \mathcal{M}_{S_n}$ is the set of all holes\footnote{With respect to the period structure given by $p_n=\lcm(S_n)$.} in $\eta$ on level $n$. (Observe also that if $\cB$ is primitive and $\eta$ is a Toeplitz sequence, then $\cB$ is taut, see \cite[Lemma~3.7]{KKL2016}.)
\end{proposition}
\begin{proof}
(i)\; Let $s\in\cM_S$.
Then $b \mid s$ for some $b \in S$. Since $b \mid \ell_{S}$, $s+\ell_{S} \mathbb{Z} \subseteq b \mathbb{Z} \subseteq \mathcal{M}_{S}$. \\
(ii)\; Let $s+\ell_S\Z\subseteq\cM_\cB$. By \cite[Proposition~4.31]{BKKL2015} the tautness of $\cB$ implies $b \mid \gcd(s,\ell_S)$ for some $b \in\cB$. Since $S$ is saturated, $b\in S$. So $s \in \mathcal{M}_{S}$.\\
(iii)\; Let $s\in\cF_{\cA_S}$. Assume for a contradiction that $\gcd(b,\ell_S) \mid s+\ell_{S} k$ for some $k \in \Z$ and some $b \in \cB$. Since
$\gcd(b,\ell_S) \mid \ell_{S}$, $\gcd\left(b, \ell_{S}\right) \mid s$, which contradicts $s\in\cF_{\cA_S}$. So $s+\ell_S\Z\subseteq\cF_{\cA_S}$.\\
(iv)\; Let $s+\ell_S\Z\subseteq\cF_\cB$. Assume for a contradiction that $\gcd(b,\ell_S)\mid s$ for some $b\in\cB$. Then there is $x\in\Z$ such that $x\equiv0\mod b$ and $x\equiv s\mod\ell_S$, i.e.~$x\in b\Z\cap(s+\ell_S\Z)$, in contradiction to $s+\ell_S\Z\subseteq\cF_\cB$. Hence $s\in\cF_{\cA_S}$.

Now a) follows from (i) and (ii), while b) follows from (iii) and (iv).
\end{proof}
To describe the set of essential holes we extract special elements of $\cA_S$.
\begin{definition} Let $S\subset\cB$.
\begin{compactenum}[a)]
\item An element $b\in\cB\setminus S$ is a \emph{source} of an element $a\in\cA_S$ if $a=\gcd(b,\lcm(S))$.
\item
$\cA_S^\infty:=\left\{a\in\cA_S: a\text{ has infinitely many sources}\right\}$, and $\cA_S^{\infty,p}:=(\cA_S^\infty)^{prim}$.
\end{compactenum}
\end{definition}
We will use some basic properties of $\cA_S^\infty$.
\begin{lemma}\label{lemma:A_S-infty}
\begin{compactenum}[a)]
\item $\cA_S^{\infty,p}\subseteq\cA_S^\infty\subseteq\cA_S\setminus \cM_S$.
\item Let $S\subset S'\subset\cB$ and $a\in\cA_S^\infty$. There exists at least one $a'\in \cA_{S'}^\infty$ such that $a=\gcd(a',\ell_S)$.
\item Let $S\subset S'\subset\cB$ and $a'\in\cA_{S'}^\infty$. Then $\gcd(\ell_S,a')\in \cA^{\infty}_S$.
\item Let $S\subset S'\subset\cB$ and $a\in\cA_S^{\infty,p}$. There exists at least one $a'\in \cA_{S'}^{\infty,p}$ such that $a=\gcd(a',\ell_S)$. In particular $\card{\cA_{S'}^{\infty,p}}\geqslant\card{\cA_S^{\infty,p}}$.
\item Let $S\subset S'\subset\cB$ with $\card{\cA_{S'}^{\infty,p}}=\card{\cA_S^{\infty,p}}$ and $a'\in\cA_{S'}^{\infty,p}$. Then $\gcd(a',\ell_S)\in\cA_S^{\infty,p}$. (Without the extra assumption this need not hold, see Example~\ref{ex:not-always-GH}.)
\end{compactenum}
\end{lemma}
\begin{proof}
a)\; Let $a\in\cA_S^\infty$. If $a=\gcd(b,\ell_S)\in \cM_S$ for some $b\in\cB$, then
there is $b'\in S$ such that $b'\mid a\mid b$, and the primitivity of $\cB$ implies $b=a=b'\in S$. This contradicts to $a$ having infinitely many sources.
\\
b)\; Let $a\in\cA_S^\infty$. There are infinitely many $b\in\cB\setminus S'$ such that $a=\gcd(b,\ell_S)$. Let $a_b':=\gcd(b,\ell_{S'})$ for these $b$. Then $\gcd(a'_b,\ell_S)=\gcd(b,\ell_S)=a$ for all these $b$, and as $\cA_{S'}$ is finite, there exists some $a'\in\cA_{S'}$ such that $a'=a'_b$ for infinitely many of them. Hence $a'\in\cA_{S'}^\infty$.\\
c)\; Clear. \\
d)\; In the situation of b) suppose that $a\in\cA_S^{\infty,p}$ and that there is $a'_0\in\cA_{S'}^{\infty,p}$ such that $a'_0\mid a'$. Then $a_0:=\gcd(a'_0,\ell_S)\in \cA_S^\infty$ and $a_0\mid\gcd(a',\ell_S)=a$. As $a\in\cA_S^{\infty,p}$, this implies $a_0=a$, so that $\gcd(a_0',\ell_S)=a$.\\
e)\; This follows from d).
\end{proof}
\begin{lemma}\label{lemma:A_S-infty-boundary}
Suppose that $S_1\subseteq S_2\subseteq\dots\nearrow\cB$ is a filtration by saturated sets. Then, for each $N\in\N$,
\begin{equation}\label{eq:A_S-boundary-infty-1}
r\in\Z\ ,\; U_{S_N}(\Delta(r))\cap\partial W\neq\emptyset
\quad\Rightarrow\quad
r\in\cM_{\cA_{S_N}^\infty}\setminus\cM_{S_N}
\end{equation}
and, for all $r\in\Z$,
\begin{equation}\label{eq:A_S-boundary-infty-2}
\begin{split}
U_{S_N}(\Delta(r))\cap\partial W\neq\emptyset
\quad\Leftrightarrow\quad
\forall
n\geqslant N: (r+\ell_{S_N}\Z)\cap\left(\cM_{\cA_{S_n}^\infty}\setminus\cM_{S_n}\right)\neq\emptyset.
\end{split}
\end{equation}
More precisely, if $h\in U_{S_N}(\Delta(r))\cap\partial W$, then
$h_{S_n}\in (r+\ell_{S_N}\Z)\cap\left(\cM_{\cA_{S_n}^\infty}\setminus\cM_{S_n}\right)$ for all $n\geqslant N$.
\end{lemma}

\begin{proof}
Suppose there exists some $h\in U_{S_N}(\Delta(r))\cap\partial W$. Then
$h_{S_n}\in (r+\ell_{S_N}\Z)\cap\left(\cM_{\cA_{S_n}}\setminus\cM_{S_n}\right)$ for all $n\geqslant N$, see
Lemma~\ref{lemma:newG-1} and Proposition~\ref{des_holes}. Hence $r\not\in\cM_{S_N}$, and
there are numbers $k_n\in\Z$ and $b_n\in\cB\setminus S_n$ $(n\geqslant N)$ such that
$\gcd(b_n,\ell_{S_N})\mid\gcd(b_n,\ell_{S_n})\mid h_{S_n}=r+k_n\ell_{S_N}$.
In particular $\gcd(b_n,\ell_{S_N})\mid r$ for all $n\geqslant N$, and as $\cA_{S_N}$ is a finite set, there exist $a\in\cA_{S_N}$ and infinitely many $b_{n_i}$ such that $\gcd(b_{n_i},\ell_{S_N})=a$. It follows that $a\in\cA_{S_N}^\infty$ and $a\mid r$.
This proves \eqref{eq:A_S-boundary-infty-1}.

Observe that, trivially, $h\in U_{S_n}(\Delta(h_{S_n}))\cap\partial W$ for each $n\geqslant N$. Hence we can apply \eqref{eq:A_S-boundary-infty-1} to $n$ and $h_{S_n}$ instead of $N$ and $r$, respectively. It follows that $h_{S_n}\in\cM_{\cA_{S_n}^\infty}\setminus\cM_{S_n}$. As $h_{S_n}=r+k_n\ell_{S_N}$, this proves the ``$\Rightarrow$''-direction of \eqref{eq:A_S-boundary-infty-2} and also the final claim.

The ``$\Leftarrow$''-implication of \eqref{eq:A_S-boundary-infty-2} follows from Lemma~\ref{lemma:newG-2} and Proposition~\ref{des_holes}, because $\cA_{S_n}^\infty\subseteq\cA_{S_n}$ for all $n\in\N$.
\end{proof}
\begin{corollary}\label{coro:Heilbronn-Rohrbach}
Suppose that $S_1\subseteq S_2\subseteq\dots\nearrow\cB$ is a filtration by saturated sets. Then, for each $n\in\N$,
\begin{equation*}
m_G(\partial W)\le d(\cM_{\cA_{S_n}^{\infty,p}})\le
1-\prod_{a\in\cA_{S_n}^{\infty,p}}\left(1-\frac{1}{a}\right).
\end{equation*}
\end{corollary}
\begin{proof}
For $n\in\N$ denote by $\cU_n$ the family of all sets $U_{S_n}(\Delta(r))$ that have non-empty intersection with $\partial W$ and by $\bigcup\cU_n$ the union of these sets. Then~\eqref{eq:A_S-boundary-infty-1} implies
$m_G(\partial W)
\le
m_G(\bigcup\cU_n)
=
\frac{\card{\cU_n}}{\ell_{S_n}}
\le
d(\cM_{\cA_{S_n}^\infty}\setminus \cM_{S_n})
\le
d(\cM_{\cA_{S_n}^{\infty,p}})$ for all $n$, and the second inequality is the Heilbronn-Rohrbach inequality \cite[Theorem~0.9]{Hall1996}.
\end{proof}
\begin{remark}\label{rem:cHn_inclusions}
Fix the period structure given by $p_n=\lcm(S_n)$. Then
\begin{equation*}
\wcH_n\subseteq\cM_{\cA_{S_n}^\infty}\setminus\cM_{S_n}\subseteq\cH_n=\cM_{\cA_{S_n}}\setminus\cM_{S_n}
\end{equation*}
by Proposition~\ref{des_holes},
Lemma~\ref{lemma:newG-2}a) and equation~\eqref{eq:A_S-boundary-infty-1} of
Lemma~\ref{lemma:A_S-infty-boundary}. Moreover, equation~\eqref{eq:A_S-boundary-infty-2} of
Lemma~\ref{lemma:A_S-infty-boundary} shows that
$r\in\wcH_N$ if and only if\;
$\forall
n\geqslant N: (r+\ell_{S_N}\Z)\cap\left(\cM_{\cA_{S_n}^\infty}\setminus\cM_{S_n}\right)\neq\emptyset$. This characterization is also the starting point for verifying
the structural assumption~\eqref{eq:additional-structure} on $\wcH_n$ from Subsection~\ref{subsec:additional-structure}, see Proposition~\ref{prop:char_ess_hol} and Theorem~\ref{theo:ess-holes-arithmetic} below.
\end{remark}

\begin{remark}
In the general Toeplitz case one can easily construct a Toeplitz sequence for which not all holes are essential. We construct a $\cB$-free Toeplitz subshift with this property in Example~\ref{ex:not_all_holes_new}. Moreover, the property $\wcH_n=\cH_n$ may depend on the choice of the period structure as we show in Example~\ref{ex:holes}, and a $\cB$-free Toeplitz subshift for which $\wcH_n\subsetneq\cM_{\cA_{S_n}^\infty}\setminus\cM_{S_n}$ is provided in Example~\ref{ex:not-always-GH}.
\end{remark}

In the rest of this subsection we show that the structural assumption \eqref{eq:additional-structure} of Propositions~\ref{prop:sufficient-for-(D')} and~\ref{prop:local-invariance-new-2} and of Theorem~\ref{theo:pre-B-free} are satisfied in the $\cB$-free setting.

\begin{definition}
Let $(S_n)_n$ be a filtration of $\cB$ by finite sets.
An integer sequence $(a_n)_{n\ge N}$ is called an \emph{$(a,\cA)$-sequence}, if $a_N=a$ and if $a_n\in\cA_{S_n}{\setminus S_n}$
and $\gcd(a_{n+1},\ell_{S_n})=a_n$ for all $n\ge N$.
\end{definition}
\begin{remark}\label{remark:a_in_infinity}
If $(a_n)_{n\ge N}$ is an $(a,\cA)$-sequence, then
$a_n\in\cA_{S_n}^\infty$ for all $n\ge N$. Indeed, suppose that~$a_m$ has only finitely many sources $b_1,\ldots,b_k$ for some $m\geq N$. Consider $n\geq m$ such that $b_1,\ldots,b_k\in S_n$. Since $a_n\not\in S_n$, there exists $b\in\cB\setminus S_n$ such that $a_n=\gcd(b,\ell_{S_n})$. Then $a_m=\gcd(a_n,\ell_{S_m})=\gcd(b,\ell_{S_m})$. So $b$ is a source of $a_m$ different from $b_1,\ldots,b_k$. This yields a contradiction.
Note also that by Lemma~\ref{lemma:A_S-infty} b), for every $a\in\cA_{S_N}^{\infty}$ there exists an $(a,\cA)$-sequence.
\end{remark}
\begin{proposition}\label{prop:char_ess_hol}
Let $(S_n)_n$ be a filtration of $\cB$ by finite sets. Then, for all $N\in\N$,
the set $\wcH_N$ is the union of sets $a\Z\setminus\cM_{S_N((a_n)_{n\ge N})}$, where the union extends over all $(a,\cA)$-sequences $(a_n)_{n\ge N}$ and where
\begin{equation}\label{eq:S_N(..)-def}
S_N((a_n)_{n\ge N})=\left\{\gcd(b,\ell_{S_N}): b\in\cB\text{ and }  b\mid a_n\vee\ell_{S_N}\text{ for some }n\ge N\right\}.
\end{equation}
\end{proposition}

\begin{proof}
Notice that in view of
Remark~\ref{rem:cHn_inclusions}
\begin{equation*}
r\in\wcH_N\;\Leftrightarrow\;
\forall
n\geqslant N: (r+\ell_{S_N}\Z)\cap\left(\cM_{\cA_{S_n}^{\infty}}\setminus\cM_{S_n}\right)\neq\emptyset.
\end{equation*}
Hence
\begin{equation*}
\begin{split}
r\not\in\wcH_N
&\Leftrightarrow\;
\exists
 n\ge N: (r+\ell_{S_N}\Z)\cap\left(\cM_{\cA_{S_n}^{\infty}}\setminus\cM_{S_n}\right)=\emptyset\\
&\Leftrightarrow\;
\exists n\ge N\ \forall a_n\in\cA_{S_n}^{\infty}:
(r+\ell_{S_N}\Z)\cap a_n\Z\subseteq\cM_{S_n}\\
&\Leftrightarrow\;
\exists n\ge N\ \forall a_n\in\cA_{S_n}^{\infty}:
(r+\ell_{S_N}\Z)\cap a_n\Z\neq\emptyset\;\Rightarrow\;\exists b\in S_n:
(r+\ell_{S_N}\Z)\cap a_n\Z\subseteq b\Z\\
&\Leftrightarrow\;
\exists n\ge N\ \forall a_n\in\cA_{S_n}^{\infty}:
\gcd(a_n,\ell_{S_N})\mid r\;\Rightarrow\;\exists b\in S_n: b\mid a_n\vee\ell_{S_N}\text{ and }\gcd(b,\ell_{S_N})\mid r.
\end{split}
\end{equation*}
The third equivalence follows from Lemma \ref{finite_mult} and the last one from Lemma \ref{arith_char_incl}.
Now the first of the following two equivalences is immediate:
\begin{eqnarray}
&&r\in\wcH_N\notag\\
&\Leftrightarrow&
\forall n\ge N\;\exists a_n\in\cA_{S_n}^{\infty}:
\gcd(a_n,\ell_{S_N})\mid r\;\text{and}\;
\left[\forall b\in S_n:
b\mid a_n\vee\ell_{S_N}\Rightarrow\gcd(b,\ell_{S_N})\nmid r\right]\notag\\
&\Leftrightarrow&
\exists a\in\cA_{S_N}^{\infty}\ \exists\text{ an $(a,\cA)$-sequence }(a_n)_{n\ge N}:
a\mid r\text{ and }\label{eq:ess-equiv-in-proof}\\
&&\hspace*{50mm}
\forall b\in\cB\ \forall n\ge N:
b\mid a_n\vee\ell_{S_N}\Rightarrow\gcd(b,\ell_{S_N})\nmid r.\notag
\end{eqnarray}
The ``$\Leftarrow$''-direction of the second equivalence is obvious - just a matter of notation. For the ``$\Rightarrow$''-direction we construct a suitable $(a,\cA)$-sequence $(a_m')_{m\ge N}$ from the given numbers $a_n$:
There is $a\in\cA_{S_N}^\infty$ such that $a=\gcd(a_n,\ell_{S_N})$ for infinitely many indices $n$. Obviously $a\mid r$, and we choose $a_N'=a$.
Suppose inductively that suitable $a_N',\dots,a_m'$ are constructed in such a way that $a_m'=\gcd(a_{n},\ell_{S_m})$ for
infinitely many different indices $n\ge m$. Then there is
an increasing subsequence $(a_{n_i})_i$ such that $\gcd(a_{n_i},\ell_{S_{m+1}})$ is the same value for all $n_i\ge m+1$. This common value is denoted $a_{m+1}'$. It satisfies $\gcd(a_{m+1}',\ell_{S_m})=\gcd(a_{n_i},\ell_{S_{m+1}},\ell_{S_m})=\gcd(a_{n_i},\ell_{S_m})=a_m'$ for all $n_i$.
Suppose now that $b\in\cB$, $n\ge N$ and $b\mid a_n\vee\ell_{S_N}$. Fix $n'\ge n$ such that $b\in S_{n'}$. Then $b\mid a_{n'}\vee\ell_{S_n}$, because $a_n\mid a_{n'}$, and we conclude that $\gcd(b,\ell_{S_N})\nmid r$.
The claim follows.
\end{proof}

\begin{theorem}\label{theo:ess-holes-arithmetic}
Let $(S_n)_n$ be a filtration of $\cB$ by finite sets and let $N>0$. For each $a\in\cA_{S_N}^\infty$ there exists a finite primitive set $S_N(a)$ of positive integers such that
\begin{equation}
\wcH_N=\bigcup_{a\in\cA_{S_N}^\infty}a\Z\setminus\cM_{S_N(a)}
\end{equation}
and all sets $a\Z\setminus\cM_{S_N(a)}$ are non-empty.
(An explicit construction of the sets $S_N(a)$ is given in the proof.) In particular, assumption \eqref{eq:additional-structure} of Theorem~\ref{theo:pre-B-free} is satisfied.
\end{theorem}
\begin{proof}
Consider any fixed $a\in\cA_{S_N}^\infty$. In view of Proposition~\ref{prop:char_ess_hol}, the set $S_N(a)$ must be constructed in such a way that $\cM_{S_N(a)}=\bigcap\cM_{S_N((a_n)_{n\ge N})}$ where the intersection runs over all $(a,\cA)$-sequences with $a_N=a$.
As all sets $S_N((a_n)_{n\ge N})$ consist of divisors of $\ell_{S_N}$, this set is only a finite intersection, say of $r$ sets of multiples $\cM_{R_1},\dots,\cM_{R_r}$. Hence we may choose as $S_N(a)$ the primitivization of the set of all $c_1\vee\dots\vee c_r$ where $c_i\in R_i$ for $i=1,\dots,r$.

Suppose for a contradiction that $a\Z\subseteq \cM_{S_N(a)}$. Then
$a\Z\subseteq\cM_{S_N((a_n)_{n\ge N})}$ for each
$(a,\cA)$-sequence $(a_n)_{n\ge N}$ with $a_N=a$. Hence, for each such sequence, there is $b\in\cB$ such that
$\gcd(b,\ell_{S_N})\mid a=\gcd(a_n,\ell_{S_N})$ and $b\mid a_n\vee\ell_{S_N}$ for some $n\ge N$.
It follows that $b^{\div\ell_{S_N}}\mid(a_n\vee\ell_{S_N})^{\div\ell_{S_N}}=a_n^{\div\ell_{S_N}}$, so that
$b=b^{\div\ell_{S_N}}\cdot\gcd(b,\ell_{S_N})\mid a_n^{\div\ell_{S_N}}\cdot \gcd(a_n,\ell_{S_N})=a_n$. But this is impossible, because $a_n\in\cA_{S_n}^\infty$ and $\cB$ is primitive.
\end{proof}

\begin{remark}\label{remark:sets_multiples}
\begin{compactenum}[a)]
\item Each set $S_N((a_n)_{n\ge N})$ contains the set $S_N$. Hence each of the sets $\cM_{S_N(a)}$ contains $\cM_{S_N}$.
\item If $\sup_N\card{\cA_{S_N}^\infty}<\infty$, then, for sufficiently large $N$, there is at most one $(a,\cA)$-sequence $(a_n)_{n\ge N}$
for each $a\in\cA_{S_N}^\infty$. Hence $\cM_{S_N(a)}=\cM_{S_N((a_n)_{n\ge N})}$ for each such $a$.
\end{compactenum}
\end{remark}

Under a special assumption (which is satisfied in all our examples except Example~\ref{ex:not-always-GH}) we have a simplified description of the sets $\wcH_n$.
\begin{lemma}\label{lemma:essential_holes_GH}
Assume that $(S_n)$ is a filtration of $\cB$ by finite saturated sets. Let $N\in\N$.
The following conditions are equivalent:
\begin{compactenum}[a)]
\item $\ell_{S_N}\vee a'\in\cF_{\cB\setminus S_N}$ for every $n>N$ and $a'\in\cA_{S_n}^{\infty}$,
\item $S_N((a_n)_{n\geq N})=S_N$ for every $a\in\cA_{S_N}^{\infty}$ and for all $(a,\cA)$-sequences $(a_n)_{n\ge N}$.
\end{compactenum}
If this the case, $\wcH_N=\cM_{\cA_{S_N}^{\infty}}\setminus\cM_{S_N}$.
\end{lemma}
\begin{proof}
Assume a) and let $a\in \cA_{S_N}^{\infty}$. Let $\gcd(b,\ell_{S_N})\in S_N((a_n)_{n\geq N})$ for some $b\in\cB$ and some $(a,\cA)$-sequence $(a_n)_{n\ge N}$.
Then $b\mid a_n\vee \ell_{S_N}$ for some $n\ge N$ and, as $a_n\in\cA_{S_n}^{\infty}$ by Remark~\ref{remark:a_in_infinity}, a) applied to $a'=a_n$ yields $b\in S_N$.

Conversely, assume that $b\mid\ell_{S_N}\vee a'$ for some $b\in\cB\setminus S_N$ and $a'\in\cA_{S_{n_0}}^{\infty}$, where $n_0> N$. There exists an $(a',\cA)$-sequence $(a_n)_{n\ge n_0}$  (see Remark~\ref{remark:a_in_infinity}).  Set $a_n=\gcd(a_{n_0},\ell_{S_n})$ for $N\le n\le n_0$. Then $(a_n)_{n\ge N}$ is an $(a_N,\cA)$-sequence such that $a_{n_0}=a'$. Moreover, $\gcd(\ell_{S_N},b)\in S((a_n)_{n\ge N})$ and $\gcd(\ell_{S_N},b)\notin S_N$ as $S_N$ is saturated. It follows that $S_N((a_n)_{n\geq N})\neq S_N$.

The remaining assertion follows by Proposition~\ref{prop:char_ess_hol} and Remark~\ref{remark:a_in_infinity}.
\end{proof}

\subsection{Trivial centralizer}\label{subsec:trivial-centralizer}

We start with an example for which the simple Proposition~\ref{prop:invariance-new} guarantees a trivial centralizer:
\begin{example}\label{ex:power:identity}
Let $\cB=\{2^nc_n: n>0\}$, where $(c_n)_n$ is a pairwise coprime sequence of odd integers. We will show that the corresponding $\cB$-free system has a trivial centralizer. Let $S_n=\{2^kc_k: 0<k\le n\}$. The sets $S_n$ form a filtration by finite sets. Then $\ell_{S_n}=2^n\prod_{i=1}^nc_i$ and
$\cA_{S_n}^{\infty,p}=\cA_{S_n}\setminus S_n=\{2^n\}$. Notice that \eqref{eq:trivial-intersection-new} is satisfied for $A_n=\cA_{S_n}^{\infty}$.
Moreover, for each $N>0$ there is only one $(a,\cA)$-sequence with $a\in\cA_{S_N}^{\infty,p}$, namely the sequence $(2^n)_{n\ge N}$, and $S_N((2^n)_{n\ge N})=S_N$ according to \eqref{eq:S_N(..)-def}. Hence $\wcH_N=2^N\setminus\cM_{S_N}=\cH_N$.
 So $\wq_n=\ell_{S_n}$. By Proposition~\ref{prop:invariance-new}, the centralizer is trivial, as shown previously in \cite{Bartnicka2017}.
\end{example}

Next we apply Theorem \ref{theo:pre-B-free} to examples which violate
conditions \eqref{eq:trivial-intersection-new} and \eqref{Seh}, hence also \eqref{D}.
\begin{example}\label{ex:2^i3^i}
Let $\cB=\{2^nc_n,3^nd_n: n>0\}$, where $(c_n)_n$ and $(d_n)_n$ are two sequences of integers coprime to $2$ and $3$ and such that the sequence $(c_n\vee d_n)_{n>0}$ is pairwise coprime. We will show that the corresponding $\cB$-free system has a trivial centralizer.
Let $S_n=\{2^kc_k,3^kd_k: 0<k\le n\}$. The sets $S_n$ form a filtration by finite sets. Notice
that $\ell_{S_n}=6^n\prod_{i=1}^n(c_i\vee d_i)$ and
$\cA_{S_n}^{\infty,p}=\cA_{S_n}\setminus S_n=\{2^n,3^n\}$.
In particular, $\gen{\cA_{S_n}^\infty}=\Z$ for all $n$, so that condition
\eqref{eq:trivial-intersection-new} is violated. Below we show that also \eqref{D} is violated, while Proposition~\ref{prop:local-invariance-new-2} shows that \eqref{D'} and \eqref{DD'} are satisfied.

We claim $\wcH_n=\cH_n=\cM_{\cA_{S_n}^{\infty,p}}\setminus\cM_{S_n}$. Indeed, suppose that $b\mid s^{n}\vee\ell_{S_N}$ for some $b\in\cB$ and some $s\in\{2,3\}$. Since $c_j\nmid s^{n}\vee\ell_{S_N}$ and $d_j\nmid s^{n}\vee\ell_{S_N}$ for any $j>N$, we have $b\in S_N$. So a)~from Lemma~\ref{lemma:essential_holes_GH} holds, and the claim follows observing also Remark~\ref{rem:cHn_inclusions}.

In view of Theorem~\ref{theo:ess-holes-arithmetic}, condition \eqref{eq:additional-structure} is satisfied, so that
Theorem~\ref{theo:pre-B-free} implies $\card K=1$ or $\card K=2$. Below we will rule out the second possibility.
\begin{compactitem}[$\bullet$]
\item If $K=\{k\}$, then $\wq_n\mid (y_F)_n-k$ for all $n>0$, where $\wq_n$ is the minimal period of $\wcH_n$, in this case the minimal period of $\cM_{\cA_{S_n}^\infty}\setminus\cM_{S_n}$. Since $S_n^{\div 2^n}=\{c_k,3^kd_k: 0<k\leq n\}$ and $S_n^{\div 3^n}=\{2^kc_k,d_k: 0<k\leq n\}$ are primitive, Lemma~\ref{lem:describe_period}b) applies.
So $\wq_n=\lcm(\cA_{S_n}^\infty\cup S_n)=\lcm(S_n)$, and the triviality of the centralizer follows from Theorem~\ref{theo:pre-B-free}.
\item Suppose for a contradiction that $\card K=2$, say $K=\{k_2=\kappa_n(2^n),k_3=\kappa_n(3^n)\}$.  Then $2^n\mid(y_F)_n-k_2$ and $3^n\mid(y_F)_n-k_3$ for all $n>0$.
For $s\in\{2,3\}$ the sets $\wcH_n^{k_s}$ (defined in Remark~\ref{remark:Hnk}a) and  $s^n\Z\cap\wcH_n^{k_s}$ are invariant under translation by $\gcd((y_F)_n-k_s,p_n)$, see Remark~\ref{remark:Hnk}c), because $s^n\mid p_n$ and $s^n\mid (y_F)_n-k_s$. Hence the same is true for the set
$\wcH_n^{k_s}\setminus s^n\Z$.

We claim that $\wcH_n^{k_s}\subseteq s^n\Z\cap\wcH_n$ for $s=2,3$. Indeed, if this is not the case, then for at least one $s\in\{2,3\}$ and $\bar s=5-s$,
\begin{equation*}
\emptyset\neq\wcH_n^{k_s}\setminus s^n\Z\subseteq\wcH_n\setminus s^n\Z\subseteq (2^n\Z\cup 3^n\Z)\setminus s^n\Z\subseteq \bar s^n\Z,
\end{equation*}
so that $\gcd((y_F)_n-k_s,p_n)$ must be a multiple of $\bar s^n$. It follows that $\bar s^n\mid(y_F)_n-k_s$. Since we observed above that $\bar s^n\mid(y_F)_n-k_{\bar s}$, we see that $\bar s^n\mid k_s-k_{\bar s}=\pm(k_2-k_3)$ for all $n>0$, which implies $k_2=k_3$ in contradiction to $\card K=2$.

Moreover, as $\wcH_n=\wcH_n^{k_s}\cup\wcH_n^{k_{\bar s}}$,
\begin{equation*}
\wcH_n\setminus\bar s^n\Z\subseteq\wcH_n^{k_s}\subseteq s^n\Z\cap\wcH_n\quad\text{for }s=2,3,
\end{equation*}
Since we proved above that $\wcH_n=\cM_{\cA_{S_n}^{\infty,p}}\setminus\cM_{S_n}=\cM_{\{2^n,3^n\}}\setminus\cM_{S_n}$, this implies
\begin{equation*}
s^n\Z\setminus\cM_{S_n\cup\{\bar s^n\}}
\subseteq\wcH_n^{k_s}\subseteq s^n\Z\setminus\cM_{S_n}\quad\text{for }s=2,3,
\end{equation*}
equivalently,
\begin{equation}\label{eq:inclusions-2-new}
s^n\Z\cap\cM_{S_n}\subseteq s^n\Z\setminus\wcH_n^{k_s}\subseteq s^n\Z\cap\cM_{S_n\cup\{\bar s^n\}}\quad\text{for }s=2,3.
\end{equation}

Denote the minimal period of $\wcH_n^{k_s}$ by $\wq^{(s)}$.
Our goal is to prove that $\gcd(\wq^{(2)},\wq^{(3)})>2m$, because the fact that $\gcd(\wq^{(2)},\wq^{(3)})$ divides
$\gcd((y_F)_n-k_2,(y_F)_n-k_{3})$ and hence also $k_{3}-k_2$ then shows that $k_2=k_3$, the desired contradiction. (Recall from Proposition~\ref{prop:sufficient-for-(D')} that $K\subseteq[-m,m]$.)

Observe first that
$s^n\mid \wq^{(s)}$ and $t+\wq^{(s)}\Z\subseteq s^n\Z\setminus\wcH_n^{k_s}\subseteq s^n\Z\cap\cM_{S_n\cup\{\bar s^n\}}$  for each $t\in s^n\Z\setminus\wcH_n^{k_s}$.
Hence there is $b\in S_n\cup\{\bar s^n\}$ such that $s^n\vee b\mid\gcd(t,\wq^{(s)})$.
\begin{compactitem}[-]
\item If the first inclusion in \eqref{eq:inclusions-2-new} is strict, there exists
$t\in (s^n\Z\setminus\wcH_n^{k_s})\setminus\cM_{S_n}$. Hence $s^n\vee\bar s^n\mid\gcd(t,\wq^{(s)})$.
\item The same arguments apply when the roles of $s$ and $\bar s$ are interchanged.
\end{compactitem}
Therefore, if $2^n\Z\cap\cM_{S_n}\subsetneq 2^n\Z\setminus\wcH_n^{k_2}$ or $3^n\Z\cap\cM_{S_n}\subsetneq 3^n\Z\setminus\wcH_n^{k_{3}}$, then $3^n\mid\gcd(\wq^{(2)},\wq^{(3)})$ or $2^n\mid\gcd(\wq^{(2)},\wq^{(3)})$, respectively, and we are done.

It remains to treat the case where
$2^n\Z\cap\cM_{S_n}= 2^n\Z\setminus\wcH_n^{k_2}$ and $3^n\Z\cap\cM_{S_n}= 3^n\Z\setminus\wcH_n^{k_{3}}$. In this case $\wq^{(s)}$ is the smallest period of $s^n\Z\setminus\cM_{S_n}$, so $\wq^{(s)}=s^n\cdot\lcm((S_n^{\div s^n})^{prim})$ for $s=2,3$ by Lemma~\ref{lem:describe_period}a), where $S_n^{\div 2^n}=\{c_i,3^id_i: 1\le i\le n\}$ and $S_n^{\div 3^n}=\{2^ic_i,d_i: 1\le i\le n\}$.
\begin{compactitem}[-]
\item If there are infinitely many $i\in\N$ such that $c_i\nmid d_i$, then there are infinitely many $n\in\N$ such that $3^n\mid \lcm((S_n^{\div 2^n})^{prim})$, so that $3^n\mid\gcd(\wq^{(2)},\wq^{(3)})$.
\item Analogously, if there are infinitely many $i\in\N$ such that $d_i \nmid c_i$, then there are infinitely many $n\in\N$ such that $2^n\mid \lcm((S_n^{\div 3^n})^{prim})$, so that $2^n\mid\gcd(\wq^{(2)},\wq^{(3)})$.
\item It remains to treat the case where $c_i=d_i$ except for finitely many $i\in\N$, say for $i<N$. Then $c_N\cdot\dots\cdot c_n\mid \gcd(\wq^{(2)},\wq^{(3)})$ for all $n\ge N$.
\end{compactitem}
\end{compactitem}
We complete this example by showing (for suitable choices of $c_n=d_n$)
that conditions \eqref{Seh} and hence also \eqref{D} are violated:
Suppose that $\prod_{n\in\N}\left(1-\frac{1}{c_n}\right)>\frac{1}{2}$. Then, for any $n\geq1$, the equation $3^nT-2^nM=1$ has solutions $T,M$ with
\begin{equation*}
3^nT,2^nM
\in
\cM_{\cA_{S_n}^\infty}\setminus\cM_{S_n}
=
\left(2^n\Z\cup3^n\Z\right)\setminus\mathcal{M}_{\{c_1,\ldots,c_n\}},
\end{equation*}
so that there are holes with distance 1 (in the sense of \cite{Bulatek1990}), see
Proposition~\ref{des_holes}.
Indeed, there is a unique solution $X_0\in\{1,\dots,6^n-1\}$ of the equations $X\equiv 1\mod 2^n$ and $X\equiv 0\mod 3^n$. For $k=0,\dots,c_1\cdots c_n-1$ let $X_k=X_0+k6^n$, $T_k=\frac{X_k}{3^n}$ and $M_k=\frac{X_k-1}{2^n}$.
As all $X_k$ are further solutions of the same two equations, and as the $c_i$ are pairwise coprime and also coprime to $2$ and to $3$, exactly $\gamma_n:=c_1\cdots c_n\prod_{i=1}^n\left(1-\frac{1}{c_i}\right)$ of these solutions are $\{c_1,\dots,c_n\}$-free, and also exactly $\gamma_n$ of the $c_1\cdots c_n$ numbers $X_k-1$ are $\{c_1,\dots,c_n\}$-free. Hence exactly $\gamma_n$ of the numbers $T_k$ and $\gamma_n$ of the numbers $M_k$ are $\{c_1,\dots,c_n\}$-free. Since $2\gamma_n>c_1\cdots c_n$, there exists at least one $k\in\{0,\ldots,c_1\cdots c_n-1\}$ such that $T_k$ and $M_k$ are $\{c_1,\dots,c_n\}$-free. This proves the claim.

\end{example}
\subsection{Non-trivial centralizer}\label{subsec:non-trivial}
The purpose of this section is to treat simple examples of the type
\begin{equation}\label{eq:ex-B_1^N}
\cB_1^N=\{2^kc_k:k\in\N\}\cup\{2^{k-1}c_k^2: k\in\N, k<N\},
\end{equation}
where $N\in\N\cup\{\infty\}$, and the numbers $c_n$ are odd and pairwise coprime. Observe that $\cB_1^1$ and $\cB_1^2$ are the sets $\cB_1$ and $\cB_1'$ from the introduction, respectively.
Let $S_n=\{2^kc_k:k\leqslant n\}\cup\{2^{k-1}c_k^2: k\leqslant\min\{n,N-1\}\}$. Then {$\lcm(S_n)=2^nc_1^2\cdots c_n^2$ for $n<N$, $\lcm(S_n)=2^nc_1^2\cdots c_{N-1}^2c_N\ldots c_n$ for $n\geq N$} and $\cA_{S_n}=S_n\cup\{2^n\}$, so that $\cA_{S_n}\setminus S_n=\cA_{S_n}^{\infty}=\cA_{S_n}^{\infty,p}=\{2^n\}$. As $\min(\cA_{S_n}\setminus S_n)$ obviously tends to infinity, these are examples of Toeplitz type by Proposition~\ref{prop:thmB}d). Along the same lines as in Example~\ref{ex:2^i3^i} one can show that all periods are essential and $\wcH_n=\cH_n=2^n\Z\setminus\cM_{S_n}$.

Let $\wq_n$ be the smallest period of $\wcH_n=\cH_n$. We will show that $\wq_n=\frac{\lcm(S_n)}{c_1\ldots c_{N-1}}$. 
Notice that $2^nk\in\cM_{S_n}$ if and only if $2^nk\in\cM_{\{c_1,\ldots,c_n\}}$. So $\wcH_n=2^n\Z\setminus\cM_{\{c_1,\ldots,c_n\}}$. Since $c_1,\ldots,c_n$ are odd and pairwise coprime, Lemma~\ref{lem:describe_period}a) implies $\wq_n=2^nc_1\ldots c_n=\frac{\lcm(S_n)}{c_1\ldots c_{N-1}}$.

Notice that \eqref{eq:trivial-intersection-new} holds for $A_n=\cA_{S_n}^{\infty,p}$. By Proposition \ref{prop:invariance-new}, there exists a unique $k\in\Z$ such that
$\partial W+y_F\subseteq \partial W+\Delta(k)$, and  $\frac{\lcm(S_n)}{c_1\ldots c_{N-1}}=\wq_n\mid (y_F)_{S_n}-k$.
It follows at once that $y_F=\Delta(k)$ for $\cB_1^1$ and that $y_F-\Delta(k)$ has order at most
$\prod_{i=1}^{N-1}c_i$ in $G$ in case of $\cB_1^N$ and finite $N$, while
$y_F-\Delta(k)$ may have any order in case of $\cB_1^\infty$.

In the remainder of this subsection we show that such nontrivial centralizers as allowed above really exist.
To this end
fix $N\in\N\cup\{\infty\}$ and $1\leq\ell<N$, and consider $\cB_1^N$ as before. Denote $p_m=\lcm(S_m)$ for any $m\in\N$.
Let
\begin{equation}
q=\frac{p_\ell}{c_\ell}=2^\ell c_1^2\ldots c_{\ell-1}^2c_\ell
\end{equation}
We define $F_\ell\colon X_\eta\to\{0,1\}^\Z$ by
\begin{equation*}
(F_\ell x)_s=
\begin{cases}
x_s&\text{ if }s\not\in c_\ell\Z-\pi_{\lcm(S_\ell)}(x),\\
x_{s+q}&\text{ if }s\in c_\ell\Z-\pi_{\lcm(S_\ell)}(x).
\end{cases}
\end{equation*}
This map is continuous, because $\pi_{\lcm(S_\ell)}$, the $\ell$-th coordinate of the MEF map, is continuous.
Recall that the MEF map $\pi:X_\eta\to G$ is chosen such that $\pi(\eta)=\Delta(0)$.
As $\Delta(0)\in C_\phi$ \cite[Lemma 3.5]{KKL2016} for $\cB$-free Toeplitz subshifts, $|\pi^{-1}\{\pi(F(\eta))\}|=|\pi^{-1}\{\pi(\eta)\}|=1$, i.e. $\pi(F(\eta))\in C_\phi$ and, observing also Remark~\ref{remark:coding-function},
\begin{equation}\label{eq:F(eta)}
F(\eta)=\phi(\pi(F(\eta)))=\phi(y_F)=1_{\Z\setminus\bigcup_{b\in\cB}(b\Z-(y_F)_b)}.
\end{equation}

\begin{lemma}
Define $y\in G$ by
\begin{equation}\label{eq:yFl}
y_b=
\begin{cases}0&\text{if}\;b\neq 2^{\ell-1}c_{\ell}^2,\\
q&\text{if}\;b= 2^{\ell-1}c_{\ell}^2.
\end{cases}
\end{equation}
Then $y$ coincides with the rotation $y_{F_\ell}$ associated to $F_\ell$.
\end{lemma}
\begin{proof}
As $\pi_{\lcm(S_\ell)}(\eta)=0$,
\begin{equation*}
(F_\ell\eta)_s
=
\begin{cases} \eta_s&\text{if}\;s\notin c_{\ell}\Z,\\
\eta_{s+q}&\text{if}\;s\in c_{\ell}\Z.
\end{cases}
\end{equation*}
In view of \eqref{eq:F(eta)} and of the injectivity of $\phi$ (see Remark~\ref{remark:coding-function}) we must show that $(F_\ell\eta)_s=0$ if and only if $s\in\bigcup_{b\in\cB}(b\Z-y_b)$ for all $s\in\Z$.\\
1) Assume that $s\notin c_{\ell}\Z$. Then
\begin{equation*}
(F_\ell\eta)_s=0\Leftrightarrow \eta_s=0 \Leftrightarrow s\in\cM_{\cB} \Leftrightarrow
s\in 2^{\ell-1}c_\ell^2\Z\cup\bigcup_{b\in\cB\setminus\{2^{\ell-1}c_\ell^2\}}(b\Z-y_b).
\end{equation*}
As $s\not\in c_\ell\Z$ by assumption and $s+q\not\in c_\ell\Z$ because $c_\ell\mid q$,
this is equivalent to
$
s\in \bigcup_{b\in\cB}(b\Z-y_b).
$
\\
2) Now assume that $s\in c_{\ell}\Z$.
Let $T=\{2c_1,\ldots,2^{\ell}c_{\ell}, c_1^2,\ldots,2^{\ell-2}c_{\ell-1}^2\}$. Observe that $q=\lcm(T)$ and $T$ is saturated. Observe also that
$(F_\ell\eta)_s=0$ if and only if $s+q\in\cM_{\cB}$.

So let $b\in\cB$ and assume that $b\mid s+q$.
\begin{compactitem}
\item[-] If $b\in T$ then $b\mid q$, hence $s\in b\Z=b\Z-y_b$.
\item[-] If $b=2^{\ell-1}c^2_{\ell}$ then  $s\in b\Z-q=b\Z-y_b$.
\item[-] If $b\in \cB\setminus(T\cup\{2^{\ell-1}c_{\ell}^2\})$ then $2^{\ell}\mid b$ and, as $c_{\ell}\mid s$ and $2^{\ell}c_{\ell}\mid q$, it follows that $s\in 2^{\ell}c_{\ell}\Z=2^{\ell}c_{\ell}\Z-y_{2^{\ell}c_{\ell}}$.
\end{compactitem}
We have proved that
$
s\in\bigcup_{b\in\cB}(b\Z-y_b).
$

It remains to prove that if  $s\in c_{\ell}\Z$ and $s\in\bigcup_{b\in\cB}(b\Z-y_b)$, then $s+q\in\cM_{\cB}$. Indeed,
\begin{eqnarray*}
s+q
&\in &
\left(c_{\ell}\Z\cap \bigcup_{b\in\cB}(b\Z-y_b)\right)+q\\
&=&
\left(\bigcup_{b\in T}(b\vee c_{\ell})\Z+q\right)\
\cup\ 2^{\ell-1}c^2_{\ell}\Z\
\cup\ \left(\bigcup_{b\in\cB\setminus (T\cup\{2^{\ell-1}c^2_{\ell}\})}c_{\ell}b\Z+q\right)\\
&\subseteq&
q\Z\ \cup\ 2^{\ell-1}c^2_{\ell}\Z\ \cup\ 2^\ell c_\ell\Z\subseteq\cM_\cB,
\end{eqnarray*}
because $\lcm(T\cup\{c_{\ell}\})=q$ and $2^{\ell}\mid b$ for each
$b\in\cB\setminus (T \cup\{2^{\ell-1}c_{\ell}^2\})$.
\end{proof}
\begin{proposition}\label{prop:non-trivial_fin_ord}
For each $1\leqslant\ell<N\leqslant\infty$, the map $F_\ell$ belongs to the centralizer of the $\cB_1^N$-free subshift. It satisfies $F_\ell^{c_\ell}=\id_{X_\eta}$, but $F_\ell^i\neq\id_{X_\eta}$ for all $1\leqslant i<c_\ell$.
\end{proposition}
\begin{proof}
We show first that $F_\ell(\eta)\in X_{\eta}$.
Recall that $\pi_{\lcm(S_\ell)}(\eta)=0$.
Fix $n\in \N$ and choose $t\geqslant\ell$ such that $2^t>n$.
As $\gcd(p_\ell,\frac{p_t}{c_\ell})=q$, there exists $z\in \Z$ such that
\begin{equation}\label{eq:kongruencje2}
z\equiv 0\mod \frac{p_t}{c_\ell}\quad\text{ and }\quad
z\equiv  q \mod p_\ell.
\end{equation}
We claim that $(F_\ell\eta)[-n,n]=\eta[-n+z,n+z]$. Let $s\in[-n,n]$. There are two cases:
\\[2mm]
1) $s\notin c_\ell\Z$.
Then $2^{\ell-1}c_\ell^2$ and $2^\ell c_\ell$ do not divide neither $s$ nor $s+z$ by the second of the congruences~\eqref{eq:kongruencje2}.
\begin{compactitem}[ -]
\item
Assume that $\eta_s=0$. Then $2^{j-\varepsilon}c_j^{1+\varepsilon}\mid s$ for some $j\neq\ell$ and $\varepsilon\in\{0,1\}$. Observe that $j\leqslant t$, since $|s|\le n<2^t$. Then $2^{j-\varepsilon}c_j^{1+\varepsilon}\mid s+z$ by the first of the congruences \eqref{eq:kongruencje2}, hence $\eta_{s+z}=0$.
\item
Conversely, assume that $\eta_{s+z}=0$. Then
$2^{j-\varepsilon}c_j^{1+\varepsilon}\mid s+z$ for some $j\neq\ell$ and $\varepsilon\in\{0,1\}$, and if $j\le t$ then $2^{j-\varepsilon}c_j^{1+\varepsilon}\mid s$
by the first of the congruences \eqref{eq:kongruencje2}, hence $\eta_s=0$. But $j>t$ is impossible, because then $2^t\mid s+z$, so that $2^t\mid s$ by the first of the congruences \eqref{eq:kongruencje2} again, in contradiction to $0<|s|<2^t$.
\end{compactitem}
We have shown that $\eta_{s+z}=\eta_s=(F_\ell\eta)_s$.
\\[2mm]
2) $s\in c_\ell\Z$. We use the fact that $z\equiv   q \mod p_\ell$ in view of \eqref{eq:kongruencje2} repeatedly.
\begin{compactitem}[ -]
\item
Assume that $\eta_{s+q}=0$, so that $2^{j-\varepsilon}c_j^{1+\varepsilon}\mid s+q$ for some $j\in\N$ and $\varepsilon\in\{0,1\}$. If $j\le\ell$ then $2^{j-\varepsilon}c_j^{1+\varepsilon}\mid s+z$ because $z\equiv   q \mod p_\ell$. If $j>\ell$ then $2^{\ell}\mid s+q$, hence $2^{\ell}\mid s$ and $2^{\ell}c_\ell\mid s$. As
$z\equiv   q \mod p_\ell$, also
$2^\ell c_\ell\mid z$, so that $2^\ell c_\ell\mid s+z$. In  both cases $\eta_{s+z}=0$.
\item
Conversely, assume that  $\eta_{s+z}=0$. The same arguments as before, with roles of $q$ and $z$ interchanged, show that $\eta_{s+q}=0$.
\end{compactitem}
We have shown that $\eta_{s+z}=\eta_{s+q}=(F_\ell\eta)_s$, thus the claim follows.

As $\pi(\eta)=\Delta(0)$, we have $\pi(\sigma^k\eta)=\Delta(k)$ for any $k\in\Z$. So $\pi_{\lcm(S_\ell)}(\sigma^k\eta)=k\mod p_\ell$ for any $k\in\Z$. Hence for any $k\in\Z$
\begin{equation*}
    (F_\ell(\sigma^k\eta))_s=\begin{cases}
(\sigma^k\eta)_s&\text{ if }s\not\in c_\ell\Z-k,\\
(\sigma^k\eta)_{s+q}&\text{ if }s\in c_\ell\Z-k
\end{cases}=\begin{cases}
\eta_{s+k}&\text{ if }s+k\not\in c_\ell\Z,\\
\eta_{s+q+k}&\text{ if }s+k\in c_\ell\Z
\end{cases}=(F_\ell\eta)_{s+k}=(\sigma^k(F_\ell\eta))_s.
\end{equation*} So $F_\ell(\sigma^k\eta)\in X_\eta$ for each $k\in\Z$. The denseness of the orbit of $\eta$ and the continuity of $F_\ell$ imply that $F_\ell(X_\eta)\subseteq X_\eta$ and $F_\ell$ commutes with $\sigma$.

Since $F_{\ell}$ corresponds to $y_{F_\ell}$ given by (\ref{eq:yFl}), and $q$ has order $c_{\ell}$ in the group $\Z/2^{\ell-1}c^2_{\ell}\Z$, it follows that $F_{\ell}$ has order $c_{\ell}$.

This proves in particular that $F_\ell$ is a homeomorphism.
\end{proof}
\begin{remark}\label{rem:finite-order}
Consider the case $N<\infty$.
As the numbers $c_1,\ldots,c_{N-1}$ are pairwise coprime, the group generated by the automophisms $F_{1},\ldots,F_{{N-1}}$ is cyclic of order $c_1\ldots c_{N-1}$. Let $F$ be a generator of this group.  Then, in view of Corollary~\ref{coro:simple-case}, $\Aut_\sigma(X_\eta)=\langle\sigma\rangle\oplus\langle F\rangle$ in the case of $\cB^{N}_1$.
\end{remark}
\begin{corollary}\label{coro:large-centralizer}
Consider $\cB_1^\infty$. Proposition~\ref{prop:non-trivial_fin_ord} shows that
the group $Aut_{\sigma}(X_{\eta})/\langle\sigma\rangle$ contains the infinite direct sum of finite cyclic groups
\begin{equation*}
\Z/c_1\Z\oplus\Z/c_2\Z\oplus\ldots\oplus\Z/c_\ell\Z\oplus\ldots\, .
\end{equation*}
\end{corollary}

\begin{remark}
One can show that the element $y\in G$ given by \eqref{eq:yFl} satisfies the sufficient conditions from \cite[Theorem 1]{Bulatek1990} for representing a (non-trivial) element of the centralizer of the $\cB_1^N$-free subshift. But the methods from \cite{Bulatek1990} do not limit the order of elements from this centralizer as in Theorem~\ref{theo:basic-case}. It is shown in~\cite{Bulatek1990} that
Toeplitz subshifts with skeletons with equidistant holes have only
elements of finite order in their centralizer. But this does not apply to minimal $\cB$-free subshifts, see Proposition~\ref{des_holes}.
\end{remark}
\subsection{Holes versus essential holes: examples}
We start with an example for which there is \emph{no period structure such that all holes are essential and the centralizer is trivial}.

\begin{example}\label{ex:not_all_holes_new}
Assume that $c_1,\ldots,c_n,\ldots,q_1,\ldots,q_n,\ldots$ are pairwise coprime natural numbers. Let
\begin{equation*}
b_1=q_1c_1,\,b_2=q_2c_2,\,b_3=q_1q_3c_3,\, b_4=q_1q_2q_4c_4,\ldots,b_m=q_1\ldots q_{m-2}q_mc_m,\ldots
\end{equation*}
and set $\cB=\{b_m:m\in\N\}$. Let $(p_n)$ be any period structure for $\eta=\eta_{\cB}$. Observe, that for every $m\in\N$ there exists $n,n'\in\N$ such that
\begin{equation*}
p_m\mid q_1\ldots q_nc_1\ldots c_n\;\text{and}\;q_1\ldots q_mc_1\ldots c_m|p_{n'}
\end{equation*}
Fix $N$ such that $q_1|p_N$ and let
\begin{equation*}
m=\max\{i\in\N:q_1\ldots q_i|p_N\}.
\end{equation*}
Then $q_{m+1}\nmid p_N$ hence
\begin{equation*}
k:=\gcd(b_{m+1},p_N)=q_1\ldots q_{m-1}c_{m+1}^{\varepsilon},
\end{equation*}
where $\varepsilon\in\{0,1\}$.
Note that
\begin{equation}\label{eq:newex2}
(k+p_N\Z)\cap q_m\Z=\emptyset.
\end{equation}
We claim that $k\in\cH_N\setminus \wcH_N$. Clearly, $k\in\cF_{\cB}$, so $\eta_k=1$. By the definition of $k$ it follows that $(k+p_N\Z)\cap b_{m+1}\Z\neq \emptyset$, thus $\eta$ is not constant along $k+p_N\Z$, hence $k\in \cH_N$.
Now take $n>N$ such that
\begin{equation}\label{eq:newex1}
q_1\ldots q_{m+1}c_1\ldots c_{m+1}|p_n.
\end{equation}
We claim that $(k+p_N\Z)\cap \cH_n=\emptyset$. Let $l\in \Z$ and assume first that $k+lp_N\in\cF_{\cB}$. Suppose that
$k+lp_N+l'p_n\in\cM_{\cB}$ for some $l'\in\Z$. Assume that $b|k+lp_N+l'p_n$ for some $b\in\cB$. By (\ref{eq:newex2}) and (\ref{eq:newex1}) $q_m\nmid b$, so $b\in\{b_1,\ldots,b_{m+1}\}$. But then, by (\ref{eq:newex1}), $b| k+lp_N$, a contradiction. It follows that $k+lp_N\notin\cH_n$.
Now assume that $k+l p_N\in\cM_{\cB}$ and  $b|k+l p_N$ for some $b\in \cB$. Then by (\ref{eq:newex2}) $b\in\{b_1,\ldots,b_{m+1}\}$, hence $b\mid p_n$ and
$k+lp_N+p_n\Z\subseteq \cM_{\cB}$ by (\ref{eq:newex1}). Again we see that $k+l p_N\notin\cH_n$. The claim follows.

Let $S_n=\{b_1,\ldots,b_n\}$. Then $p_n:=\lcm(S_n)=q_1\ldots q_nc_1\ldots {c_n}$ defines a period structure,
${\cA_{S_n}\setminus S_n}=\{q_1\ldots q_{n-1},q_1\ldots q_n\}$ and $\cA_{S_n}^\infty=\{q_1\ldots q_n\}$.
By Proposition \ref{des_holes}, with respect to this period structure, $\cH_n=q_1\ldots q_{n-1}\Z\setminus\cM_{S_n}$. Since $c_m\nmid \ell_{S_N}\vee q_1\ldots q_n$ for any $m,n>N$, $\ell_{S_N}\vee q_1\ldots q_n\in\mathcal{F}_{\cB\setminus S_N}$. So a) from Lemma~\ref{lemma:essential_holes_GH} holds. Hence $\wcH_n=q_1\ldots q_n\Z\setminus\cM_{S_n}$.
Since $S_n^{\div q_1\ldots q_n}=\{c_1,\ldots,c_n\}$ and $S_n^{\div q_1\ldots q_{n-1}}=\{c_1,\ldots,c_{n-1},q_nc_n\}$ are both
primitive,
Lemma \ref{lem:describe_period}  shows that the minimal periods of $\wcH_n$ and $\cH_n$ are both equal $p_n=q_1\ldots q_nc_1\ldots c_n$, although $\wcH_n\neq \cH_n$. Proposition~\ref{prop:invariance-new} shows that the centralizer of $X_\eta$ is trivial.
\end{example}

We continue with an example for which \emph{the validity of the identities $\wcH_n=\cH_n$ depends on the choice of the period structure.}
\begin{example}\label{ex:holes} Assume that we have a collection $\{q_i,c_i,d_i: i\ge 1\}$ of pairwise coprime odd natural numbers greater than 1. Let
\begin{equation*}
b_i=2^iq_ic_i,\; b'_i=2^iq_id_i,\; b''_i=2^{i+1}q_i
\end{equation*}
for $i\ge 1$ and
\begin{equation*}
S_n=\{b_i,b'_i, b''_i:1\le i\le n\}, \; S'_n=S_n\cup \{b_{n+1}\}
\end{equation*}
for $n\ge 1$. Finally, we set
\begin{equation*}
\cB=\bigcup_{n\ge 1}S_n= \bigcup_{n\ge 1}S'_n.
\end{equation*}
Clearly, $\cB$ contains no scaled copy of an infinite coprime set.

The sets $S_n$ and $S'_n$ are saturated and
\begin{equation}
\lcm(S_n)=2^{n+1}\cdot q_1\cdot\ldots \cdot q_n\cdot c_1\cdot\ldots \cdot c_n\cdot d_1\cdot\ldots\cdot d_n, \; \lcm(S'_n)=q_{n+1}c_{n+1}\lcm(S_n)
\end{equation}
hence
\begin{equation}\label{eq:bez_prim}
\cA_{S_n}\setminus S_n=\cA^{\infty}_{S_n}=\{2^{n+1}\}
\end{equation}
and
\begin{equation}\label{eq:prim}
\cA_{S'_n}\setminus S'_n=\{2^{n+1}, 2^{n+1}q_{n+1}\},\;   \cA^{\infty}_{S'_n}=\{2^{n+1}\}.
\end{equation}
It follows  that $\wcH_N=\cH_N=2^{N+1}\left(\Z\setminus\cM_{\{q_1,\dots,q_N\}}\right)$, where the sets of holes and essential holes are calculated with respect to the period structure $p_n=\lcm(S_n)$.
Indeed,
\begin{equation*}
\cH_N=2^{N+1}\Z\setminus\cM_{S_N}=2^{N+1}\left(\Z\setminus\cM_{(S_N^{\div 2^{N+1}})^{prim}}\right)
=2^{N+1}\left(\Z\setminus\cM_{\{q_1,\dots,q_N\}}\right)
\end{equation*}
by (\ref{eq:bez_prim}) and Proposition \ref{des_holes}.
Suppose that $b\mid\ell_{S_N}\vee 2^{n+1}$ for some $b\in\cB$ and some $n>N$. Since $q_i\nmid\ell_{S_N}\vee 2^{n+1}$ for any $i>N$, $b\in S_N$. So a) from Lemma~\ref{lemma:essential_holes_GH} holds and the assertion follows.
By Lemma~\ref{lem:describe_period}a), we obtain $\wq_N=\tau_N=2^{N+1}q_1\cdots q_N$.

Let $\cH'_n$, $\wcH'_n$ be the sets of holes and essential holes, respectively, calculated with respect to the period structure $p'_n=\lcm(S'_n)$.
We will show that
\begin{equation*}
\cH_N'=2^{N+1}\left(\Z\setminus\cM_{\{q_1,\dots,q_N,q_{N+1}c_{N+1}\}}\right)\quad\text{and}\quad
\wcH_N'=2^{N+1}\left(\Z\setminus\cM_{\{q_1,\dots,q_N,q_{N+1}\}}\right),
\end{equation*}
so that $\wq_N'=\tau_N'/c_{N+1}=2^{N+1}q_1\cdots q_{N+1}$ by Lemma~\ref{lem:describe_period}a).

Indeed,
\begin{equation*}
\cH_N'=2^{N+1}\Z\setminus\cM_{S_N'}=2^{N+1}\left(\Z\setminus\cM_{(S_N'^{\div 2^{N+1}})^{prim}}\right)
=2^{N+1}\left(\Z\setminus\cM_{\{q_1,\dots,q_N,q_{N+1}c_{N+1}\}}\right)
\end{equation*}
by (\ref{eq:prim}) and Proposition \ref{des_holes}.
Let $n>N$. (Notice that a) from Lemma~\ref{lemma:essential_holes_GH} does not hold because $2^{N+2}q_{N+1}\mid2^{n+1}\vee\lcm(S_N')$ and $2^{N+2}q_{N+1}\not\in S_N'$.) Suppose that for some $k\in\Z\setminus\cM_{\{q_1,\dots,q_N,q_{N+1}c_{N+1}\}}$
\begin{equation*}
(2^{N+1}k+\lcm(S'_N)\Z)\cap \cH'_n=\emptyset.
\end{equation*}
Then, in particular,
\begin{equation*}
(2^{N+1}k+\lcm(S'_N)\Z)\cap2^{n+1}\Z\subseteq \cM_{S_n'}.
\end{equation*}
By Lemma \ref{finite_mult}, there exists $b\in S_n'$ such that
\begin{equation*}
b\mid\gcd(2^{N+1}k,\lcm(S_N'))\vee2^{n+1}=\gcd(2^{N+1}k\vee 2^{n+1},\lcm(S_N')\vee2^{n+1}).
\end{equation*}
Since $k\in\Z\setminus\cM_{\{q_1,\dots,q_N,q_{N+1}c_{N+1}\}}$ and $b\mid 2^{N+1}k\vee2^{n+1}$, we have $b\in S_n'\setminus S_N'$. On the other hand, $b\mid\lcm(S_N')\vee2^{n+1}$ but
$q_i\nmid \lcm(S_N')\vee2^{n+1}$ for any $i>N+1$, $d_{N+1}\nmid \lcm(S_N')\vee2^{n+1}$ and $2^{N+1}q_{N+1}c_{N+1}\in S_N'$. So $b=2^{N+2}q_{N+1}$ and $q_{N+1}\mid k$.
Conversely,
\begin{equation*}
(2^{N+1}q_{N+1}m+\lcm(S'_N)\Z)\cap2^{n+1}\subseteq 2^{n+1}q_{N+1}\Z\subseteq 2^{N+2}q_{N+1}\Z\subseteq\cM_{S_n'}
\end{equation*}
for any $m\in\Z$. Hence
\begin{equation*}
\begin{split}
\wcH_N'&=2^{N+1}\Z\setminus\cM_{S_N'\cup\{2^{N+1}q_{N+1}\}}=2^{N+1}\left(\Z\setminus\cM_{((S_N'\cup\{2^{N+1}q_{N+1}\})^{\div 2^{N+1}})^{prim}}\right)\\
&=
2^{N+1}\left(\Z\setminus\cM_{\{q_1,\dots,q_N,q_{N+1}\}}\right).
\end{split}
\end{equation*}

For both filtrations property \eqref{eq:trivial-intersection-new} from Subsection~\ref{subsec:additional-structure} is obviously satisfied so that Proposition~\ref{prop:invariance-new} applies. But $p_N/\wq_N=c_1\cdots c_N\cdot d_1\cdots d_N$ and $p_N'/\wq_N'=c_1\cdots c_{N+1}\cdot d_1\cdots d_N$ are both unbounded in~$N$, so that only very weak conclusions can be drawn from this proposition. In particular, no bound on the size of the centralizer can be deduced from it.
\end{example}

Finally we provide an example for which \emph{$\wcH_n\subsetneq\cM_{\cA_{S_n}^\infty}\setminus\cM_{S_n}$ for any saturated filtration $(S_n)$ of $\cB$, so that $S_n\subsetneq S_n(a)$ for some $a\in\cA_{S_n}^{\infty}$ (hence also for some $a\in\cA_{S_n}^{\infty,p}$)}, see Theorem~\ref{theo:ess-holes-arithmetic} and Remark~\ref{remark:sets_multiples}a).
\begin{example}\label{ex:not-always-GH}
Let $\np_i, \np'_i, q_i, r_i, d_{i-1}$ $(i\in \N)$ be pairwise different primes.
Let
\begin{equation}
b_1=\np_1\cdot \np'_1\cdot q_1\cdot d_0
\end{equation}
and, for $m\in \N\cup\{0\}$ and $i\ge 2$, let
\begin{equation}
 b_{i,m}=\frac{1}{\np'_{i-1}}\np_1\cdot \np'_1\cdot\ldots \cdot \np_i\cdot \np'_i\cdot q_i\cdot r_i^m\cdot d_m.
\end{equation}
We set
\begin{equation}
\cB=\{b_1\}\cup \{b_{i,m}:i\ge 2, m\ge 0\}.
\end{equation}
Then $\cB$ is primitive.
It is easy to show that $\cB$ contains no scaled copy of an infinite coprime set, so the $\cB$-free shift is Toeplitz by Proposition~\ref{prop:thmB}.

Let $(S_n)$ be any saturated filtration of $\cB$ by finite sets. With no loss of generality we can assume that $b_1, b_{2,0}\in S_1$. Let $k$ be the minimal number such that $b_{k+1,0}\notin S_1$. It follows that $b_{2,0},\ldots,b_{k,0}\in S_1$. Let $n$ be maximal such that $b_{k+1,0}\notin S_n$. Then
 \begin{equation}
 b_1, b_{2,0},\ldots,b_{k,0}\in S_n\;\text{and}\;b_{k+1,0}\in S_{n+1}\setminus S_n
 \end{equation}
and
 \begin{equation}\label{eq:nGH11}
 \np_1\cdot \np'_1\cdot\ldots \cdot \np_k\cdot \np'_k\cdot q_1\cdot\ldots\cdot q_k\cdot d_0\mid\ell_{S_n}.
\end{equation}
  Observe that $b_{k+1,m}\notin S_n$ for every $m\in\N$. Otherwise, as $d_0|\ell_{S_n}$ and $S_n$ is saturated, $b_{k+1,0}\in S_n$ contrary to our assumption. Thus $r_{k+1}\nmid\ell_{S_n}$.
 For $m$ large enough, say for $m\ge m_0$, the number $\gcd(b_{k+1,m},\ell_{S_{n+1}})$ does not depend on $m$.

A case by case analysis of prime divisors of $b_{k+1,0}$ and $b_{k+1,m}$ shows that
 \begin{equation}\label{eq:GH12}
d_0\gcd(b_{k+1,m},\ell_{S_n})=\gcd(b_{k+1,0},\ell_{S_n})\;\text{for}\;m\ge m_0.
 \end{equation}
 Indeed, $\np_i|b_{k+1,0}$ (resp. $\np_i'|b_{k+1,0}$) if and only if $\np_i|b_{k+1,m}$ (resp. $\np_i'|b_{k+1,m}$) for every $i\in\N$. Moreover, $r_{k+1}\nmid \ell_{S_n}$, $d_0\nmid b_{k+1,m}$ and $q_{k+1}$ divides both  $b_{k+1,0}$ and $b_{k+1,m}$.   

We prove that
\begin{equation}\label{eq:an+1}
\forall{a\in\cA_{S_{n+1}}^{\infty}:\ \gcd(a,\ell_{S_n})|\gcd(b_{k+1,0},\ell_{S_n})\ \Leftrightarrow\ 
a=\gcd(b_{k+1,m},\ell_{S_{n+1}}) \;\text{for}\; m\ge m_0.}
\end{equation}
 In view of (\ref{eq:GH12}) it is enough to prove "$\Rightarrow$". We can assume that $a=\gcd(b_{i,m},\ell_{S_{n+1}})$ for some $i>1$ and $m\in\N$
and $\gcd(b_{i,m},\ell_{S_n})=\gcd(a,\ell_{S_n})\mid\gcd(b_{k+1,0},\ell_{S_n})$. As $\np'_k\nmid b_{k+1,0}$ and $\np'_k|\ell_{S_n}$, we have $i\le k-1$ or $i=k+1$. But $i\le k-1$ implies that $q_i\mid \gcd(b_{i,m},\ell_{S_n})\mid b_{k+1,0}$, a contradiction. Thus $i=k+1$ and since $a\in\cA_{S_{n+1}}^{\infty}$, $a=\gcd(b_{k+1,m},\ell_{S_{n+1}})$ for $m\ge m_0$.

Note that, as $b_{k+1,0}\in S_{n+1}$,
\begin{equation}\label{eq:nGH5}
\forall {m\in\N}:\;\np_{k+1}\cdot \np_{k+1}'\cdot q_{k+1}\mid \gcd(b_{k+1,m},\ell_{S_{n+1}}).
\end{equation}
Let $a\in\cA_{S_{n+1}}^{\infty}$ be such that $(\gcd(b_{k+1,0},\ell_{S_n})+\ell_{S_n}\Z)\cap a\Z\neq \emptyset$. Then $\gcd(a,\ell_{S_n})|\gcd(b_{k+1,0},\ell_{S_n})$ and by (\ref{eq:an+1}) $a=\gcd(b_{k+1,m},\ell_{S_{n+1}})$ for $m\ge m_0$.
By \eqref{eq:nGH11} and (\ref{eq:nGH5})
\begin{equation}
b_{k+1,0}\mid a\vee\ell_{S_n}.
\end{equation}
It follows from Lemma \ref{arith_char_incl}  that $(\gcd(b_{k+1,0},\ell_{S_n})+\ell_{S_n}\Z)\cap a\Z\subseteq b_{k+1,0}\Z$. Hence
\begin{equation*}
(\gcd(b_{k+1,0},\ell_{S_n})+\ell_{S_n}\Z)\cap(\cM_{\cA_{S_{n+1}}^{\infty}}\setminus \cM_{S_{n+1}})=\emptyset,
\end{equation*}
which, because of Remark \ref{rem:cHn_inclusions}, implies $\gcd(b_{k+1,0},\ell_{S_n})\notin \wcH_n$. Moreover, by (\ref{eq:GH12}) and the primitivity of $\cB$, $\gcd(b_{k+1,0},\ell_{S_n})\in\cM_{\cA_{S_n}^{\infty}}\setminus\cM_{S_n}$.
Thus $\gcd(b_{k+1,0},\ell_{S_n})\in\cM_{\cA_{S_n}^{\infty}}\setminus(\cM_{S_n}\cup\wcH_n)$.

We claim that $\eta$ is a regular Toeplitz sequence. Indeed, let $S_n=\{b_1\}\cup\{b_{i,m}: 2\leq i\leq n, 0\leq m\leq n\}$. Then $\ell_{S_n}=\np_1\np_1'\ldots \np_n\np_n'q_1\ldots q
_nr_1^n\ldots r_n^nd_0\ldots d_n$ and for $b_{i,m}\not\in S_n$
\begin{equation*}
\gcd(b_{i,m},\ell_{S_n})=\begin{cases}
\frac{1}{\np_{i-1}'}\np_1\np_1'\ldots \np_i\np_i'q_ir_i^n &\text{if }i\leq n,m>n,\\
\np_1\np_1'\ldots \np_{n-1}\np_{n-1}'\np_nd_m &\text{if } i=n+1,m\leq n,\\
\np_1\np_1'\ldots \np_{n-1}\np_{n-1}'\np_n &\text{if } i=n+1,m>n,\\
\np_1\np_1'\ldots \np_n\np_n'd_m &\text{if }i>n+1,m\leq n,\\
\np_1\np_1'\ldots \np_n\np_n'&\text{if }i>n+1,m>n.
\end{cases}
\end{equation*}
Hence $\cA_{S_n}^{\infty,p}=\{\frac{1}{\np_{i-1}'}\np_1\np_1'\ldots \np_i\np_i'q_ir_i^n: 2\leq i\leq n\}\cup\{\np_1\np_1'\ldots \np_{n-1}\np_{n-1}'\np_n\}$. Since $d(\cM_{\cA_{S_n}^{\infty,p}})\leq 1-\prod_{a\in\cA_{S_n}^{\infty,p}}\left(1-\frac{1}{a}\right)\leq1-(1-\frac{1}{\min_{n\geq i\geq2}{r_i}^n})^{n-1}(1-\frac{1}{\np_1\np_1'\ldots \np_{n-1}\np_{n-1}'\np_n})\to 0$ as $n\to\infty$, by Corollary \ref{coro:Heilbronn-Rohrbach} $m_G(\partial W)=0$. So the $\cB$-free Toeplitz shift is regular, see e.g.~\cite[Theorem~13.1]{Downarowicz2005}. Similarly one can show that condition \eqref{eq:theo2-ass-2-new-a} of Proposition \ref{prop:sufficient-for-(D')} and Theorem \ref{theo:pre-B-free}
is satisfied for $A_n=\cA_{S_n}^{\infty,p}$. We do not attempt to determine the sets $\wcH_n$ (and their periods) according to the prescription in Theorem~\ref{theo:ess-holes-arithmetic} explicitly.

Note that $a_{n+1}=\np_1 \np_1'\cdots \np_n \np_n' \np_{n+1}$ is an example of a number in $\cA_{S_{n+1}}^{\infty,p}$
for which $\gcd(a_{n+1},\ell_{S_n})$ belongs to $\cA_{S_n}^\infty$ but not to $\cA_{S_n}^{\infty,p}$,
compare Lemma \ref{lemma:A_S-infty} c) and e).

We complete this example by showing (for suitable choices of $d_i$)
that conditions \eqref{Seh} and hence also \eqref{D} \footnote{The equivalence of \eqref{Seh} and \eqref{D} was claimed without proof in Subsection~\ref{subsec:separated-holes-contition}, see also Lemma~\ref{lemma:weak-separated-essential-holes-contition}.} and \eqref{eq:trivial-intersection-new}, see Proposition~\ref{prop:invariance-new}, are violated:
Let $a_n=s_1s_1's_2s_2'\ldots s_{n-1}s_{n-1}'s_n$ and $a_n'=s_1s_2s_2'q_2r_2^n$. Suppose that $\prod_{n\geq0}\left(1-\frac{1}{d_n}\right)>\frac{1}{2}$. Then the equation $\frac{a_n}{s_1s_2s_2'}T-\frac{a_n'}{s_1s_2s_2'}M=1$ has solutions $T,M$ with
\begin{equation*}
a_nT,a_n'M
\in
\cM_{\cA_{S_n}^\infty}\setminus\cM_{S_n}
\end{equation*}
so that, for all $n$, there are holes with distance $s_1s_2s_2'$ (in the sense of \cite{Bulatek1990}) in $\cH_n$, see
Proposition~\ref{des_holes}.
Indeed, there is a unique solution $X_0\in\{1,\dots,\frac{\lcm(a_n,a_n')}{s_1s_2s_2'}-1\}$ of the equations $X\equiv 1\mod \frac{a_n}{s_1s_2s_2'}$ and $X\equiv 0\mod \frac{a_n'}{s_1s_2s_2'}$. For $k=0,\dots,d
_0d_1\cdots d_n-1$ let $X_k=X_0+k\frac{\lcm(a_n,a_n')}{s_1s_2s_2'}$, $T_k=\frac{X_ks_1s_2s_2'}{a_n}$ and $M_k=\frac{(X_k-1)s_1s_2s_2'}{a_n'}$.
As all $X_k$ are further solutions of the same two equations, and as the $d_i$ are pairwise different primes, exactly $\gamma_n:=d_0d_1\cdots d_n\prod_{i=0}^n\left(1-\frac{1}{d_i}\right)$ of these solutions are $\{d_0,d_1,\dots,d_n\}$-free, and also exactly $\gamma_n$ of the $d_0d_1\cdots d_n$ numbers $X_k-1$ are $\{d_0,d_1,\dots,d_n\}$-free. Hence exactly $\gamma_n$ of the numbers $T_k$ and $\gamma_n$ of the numbers $M_k$ are $\{d_0,d_1,\dots,d_n\}$-free. Since $2\gamma_n>d_0d_1\cdots d_n$, there exists at least one $k\in\{0,\ldots,d_0d_1\cdots d_n-1\}$ such that $T_k$ and $M_k$ are $\{d_0,d_1,\dots,d_n\}$-free, so then $a_nT_k$ and $a_n'M_k$ are $S_n$-free. This proves the claim.
\end{example}
\subsection{Superpolynomial complexity}\label{sec:complexity}
We consider the example $\cB=\cB_2=\{2^ic_i, 3^ic_i:i\in\N\}$, where $c_i$ are odd pairwise coprime numbers not divisible by 3.
Recall from Example~\ref{ex:2^i3^i} that our Theorem~\ref{theo:pre-B-free} applies to this $\cB$ and ensures that the $\cB$-free subshift has a trivial centralizer. Here we show that it has superpolynomial complexity.

We denote by $\rho$ the complexity function of $X_{\eta}$ for $\eta=\1_{\cF_{\cB}}$, that is
\begin{equation*}
\rho(n)=|\{\eta[k+1,k+n]:k\in \Z\}|
\end{equation*}
for $n\in\N$.

\begin{proposition}\label{proposition_complexity}
Assume that $2c_1< 2^2c_2< 2^3c_3<\ldots$ are such that
\begin{equation}\label{assumption1}
\sum_{i=1}^{\infty}\frac{1}{c_i}<\frac12,
\end{equation}
and there exists a real number $\alpha>1$ such that
\begin{equation}\label{assumption2}
c_j\le \alpha^j
\end{equation}
for $j\gg 0$.
Then, for each $\varepsilon\in(0,1)$,
\begin{equation*}
\liminf\limits_{n\rightarrow+\infty}\frac{\rho(n)}{n^{\varepsilon\lg_{2\alpha}\lg_{2\alpha}n}}=+\infty.
\end{equation*}
\end{proposition}

\begin{remark}
If $c_i$ is the square of the $(i+2)$th odd prime number, for $i\in\N$, (i.e. $c_1=25$ etc.) the assumptions of the Proposition are satisfied (\eqref{assumption2} holds for $\alpha\geq25$).
\end{remark}

\begin{remark}\label{rem:averages}
For any $j\in\N$, under the assumption (\ref{assumption1})
\begin{equation*}
c_{1}\ldots c_j \ge (2j)^{j}>j^j.
\end{equation*}
This inequality is a consequence of the fact that the arithmetic mean of positive numbers is greater or equal than their geometric mean.
\end{remark}

The following lemma is elementary.
\begin{lemma}\label{lemma1_complexity}
For every $n,b\in\N$ and $k,r\in\Z$:
\begin{equation*}
\frac{n}{b}-1<|(b\Z+r)\cap [k+1,k+n]|<\frac{n}{b}+1.
\end{equation*}
\end{lemma}

Let
\begin{equation*}
\delta=\frac12-\sum_{i=1}^{\infty}\frac{1}{c_i}.
\end{equation*}
Given $n\in\N$ let $m_n=[\lg_2n]$. Note that $\delta>0$ by (\ref{assumption1}) and $2^{m_n+1}>n$.

\begin{lemma}\label{lemma2_complexity}
Assume that $j_0\in\{1,\ldots,m_n\}$ satisfies
\begin{equation*}
 2^{j_0}c_{j_0}<\frac{\delta n}{2\lg_2n}.
\end{equation*}
If
\begin{equation*}
 [1,n]\cap (2^{j_0}c_{j_0}\Z+r)\subseteq \bigcup_{i=1}^{m_n}(2^ic_i\Z+s_i)\cup \bigcup_{i=1}^{m_n}(3^ic_i\Z+t_i)
\end{equation*}
for some $r,s_1,\ldots,s_{m_n},t_1,\ldots,t_{m_n} \in\Z$, then $r\equiv s_{j_0} \mod 2^{j_0}c_{j_0}$.
 \end{lemma}

 \begin{proof}
 Suppose otherwise, then $(2^{j_0}c_{j_0}\Z+r)$ is disjoint to $(2^{j_0}c_{j_0}\Z+s_{j_0})$ and hence
 \begin{equation}\label{formula1_complexity}
 [1,n]\cap (2^{j_0}c_{j_0}\Z+r)
 \subseteq \bigcup_{i\in\{1,\ldots,m_n\}\setminus\{j_0\}}(2^ic_i\Z+s_i)\cup \bigcup_{i=1}^{m_n}(3^ic_i\Z+t_i).
 \end{equation}
 If follows that
 \begin{equation}\label{formula2_complexity}
 \begin{split}
& |[1,n]\cap (2^{j_0}c_{j_0}\Z+r)|\\
 \le&
 \sum\limits_{i\in\{1,\ldots,m_n\}\setminus\{j_0\}}|[1,n]\cap(2^ic_i\Z+s_i)\cap (2^{j_0}c_{j_0}\Z+r)|+ \sum\limits_{i=1}^{m_n}|[1,n]\cap(3^ic_i\Z+t_i)\cap (2^{j_0}c_{j_0}\Z+r)|.
 \end{split}
 \end{equation}
 For $i\neq j_0$, $(2^ic_i\Z+s_i)\cap (2^{j_0}c_{j_0}\Z+r)$ is either empty or equal to $\lcm(2^ic_i,2^{j_0}c_{j_0})\Z+r'$ for some $r'\in\Z$. Similarly, for $i=1,\ldots,m_n$, $(3^ic_i\Z+r_i)\cap (2^{j_0}c_{j_0}\Z+r)$ is either empty or equal to $\lcm(3^ic_i,2^{j_0}c_{j_0})\Z+r''$ for some $r''\in\Z$. Then \eqref{formula2_complexity} and Lemma \ref{lemma1_complexity} yield
 \begin{equation}\label{formula3_complexity}
 \frac{n}{2^{j_0}c_{j_0}}-1\le \sum_{i\in\{1,\ldots,m_n\}\setminus\{j_0\}}\left(\frac{n}{\lcm(2^ic_i,2^{j_0}c_{j_0})}+1\right)+\sum_{i=1}^{m_n}\left(\frac{n}{3^i2^{j_0}\lcm(c_i,c_{j_0})}+1\right).
 \end{equation}
 For $i\neq j_0$, $\lcm(2^ic_i,2^{j_0}c_{j_0})=2^{\max\{i,j_0\}}c_ic_{j_0}\ge 2^{j_0}c_ic_{j_0}$ and $\lcm(c_i,c_{j_0})\ge c_{j_0}$ for every $i$, thus
 \begin{equation}\label{formula4_complexity}
 \frac{n}{2^{j_0}c_{j_0}}-1\le \sum_{i\in\{1,\ldots,m_n\}\setminus\{j_0\}}\left(\frac{n}{2^{j_0}c_ic_{j_0}}+1\right)+\sum_{i=1}^{m_n}\left(\frac{n}{3^i2^{j_0}c_{j_0}}+1\right),
 \end{equation}
 which, as $\sum_{i=1}^{m_n}\frac{1}{3^i}< \frac12$, implies
 \begin{equation}
 \frac{n}{2^{j_0}c_{j_0}}-1\le \frac{n}{2^{j_0}c_{j_0}}\left(\frac12+\sum_{i=1}^{m_n}\frac{1}{c_i}\right)+2m_n-1,
 \end{equation}
 hence
 \begin{equation}\label{formula5_complexity}
 \frac{\delta n}{2^{j_0}c_{j_0}}\le 2m_n.
 \end{equation}
 It follows that
 \begin{equation}\label{formula6_complexity}
2^{j_0}c_{j_0}\ge\frac{\delta n}{2m_n}\ge \frac{\delta n}{2\lg_2n}.
 \end{equation}
 contrary to the assumption.
 \end{proof}

Let $n\in\N$ be big enough to satisfy $2c_1<\frac{\delta n}{2\lg_2n}$, and let $j_n$ be the greatest natural number such that
 \begin{equation}\label{formula7_complexity}
 2^{j_n}c_{j_n}<\frac{\delta n}{2\lg_2n}.
 \end{equation}
 It follows by (\ref{formula7_complexity}) that $j_n\le m_n$. Moreover,
 $(j_n)$ is a non-decreasing (starting from $n$ big enough) sequence  such that $\lim_{n\rightarrow+\infty}j_n=+\infty.$

Let $N$ be a natural number such that $c_j\le \alpha^j$ for every $j\ge j_N$ and $2c_1<\frac{\delta N}{2\lg_2N}$.
If $n\ge N$, then
 \begin{equation}\label{ineq_j_n}
 j_n\ge \lg_{2\alpha}\left(\frac{\delta n}{2\lg_2n}\right)-1.
 \end{equation}
 Indeed, otherwise
 \begin{equation}
 2^{j_n+1}c_{j_n+1}\le (2\alpha)^{j_n+1}<\frac{\delta n}{2\lg_2n},
 \end{equation}
 a contradiction with the choice of $j_n$.

 \begin{lemma}\label{lemma3_complexity}
 For any sequence ${\bf r}=(r_1,\ldots,r_{m_n})$ there exists $x_{\bf r}\in\Z$ such that
\begin{equation*}
 \left\{\begin{array}{lll}
 x_{\bf r}\equiv 2^jr_j&\mod & 2^jc_j\;\;\text{for}\;j=1,\ldots,m_n,\\
 x_{\bf r}\equiv 0&\mod& 2^{m_n+1}3^{m_n+1}.
 \end{array}\right.
\end{equation*}
 Moreover, if ${\bf r'}=(r'_1,\ldots,r'_{m_n})$ is another sequence of integers and
 \begin{equation}\label{formula8_complexity}
 \eta[x_{\bf r}+1,x_{\bf r}+n]=\eta[x_{\bf r'}+1,x_{\bf r'}+n],
 \end{equation}
 where $x_{\bf r'}$ is defined analogously, then
\begin{equation*}
 r_j\equiv r'_j\mod c_j
\end{equation*}
 for $j\le j_n$, provided $n\ge N$.
 \end{lemma}

 \begin{proof}
 The existence of $x_{\bf r}$ follows by CRT. Assume (\ref{formula8_complexity}).
  Since $2^{m_n+1}3^{m_n+1}\mid x_{\bf r}$, $2^{m_n+1}3^{m_n+1}\mid x_{\bf r'}$ and $2^{m_n+1}>n$, the sets $[x_{\bf r}+1,x_{\bf r}+n]$ and $[x_{\bf r'}+1,x_{\bf r'}+n]$ are disjoint to $\bigcup_{i> m_n}(2^ic_i\Z\cup 3^ic_i\Z)$. Therefore (\ref{formula8_complexity}) implies
   \begin{equation*}
  [1,n]\cap\bigcup_{i=1}^{m_n}((2^ic_i\Z\cup 3^ic_i\Z)-x_{\bf r})=[1,n]\cap\bigcup_{i=1}^{m_n}((2^ic_i\Z\cup 3^ic_i\Z)-x_{\bf r'}).
  \end{equation*}
  In particular,
\begin{equation*}
  [1,n]\cap(2^jc_j\Z-x_{\bf r})\subseteq \bigcup_{i=1}^{m_n}((2^ic_i\Z\cup 3^ic_i\Z)-x_{\bf r'})
\end{equation*}  
for every $j=1,\ldots,m_n$. If $j\le j_n$, then $2^jc_j<\frac{\delta n}{2\lg_2n}$, and by Lemma \ref{lemma2_complexity} we conclude that
  $x_{\bf r}\equiv x_{\bf r'}\mod 2^jc_j$. As $x_{\bf r}\equiv 2^jr_j \mod 2^jc_j$ and $x_{\bf r'}\equiv 2^jr_j'\mod 2^jc_j$, it follows that $r_j\equiv r'_j\mod c_j$.
 \end{proof}

 \begin{proof}[Proof of Proposition \ref{proposition_complexity}]
Take $n\geq N$ big enough. By Lemma \ref{lemma3_complexity}, to every sequence ${\bf r}=(r_1,\ldots,r_{m_n})$ of integers we can associate a block\footnote{The choice is not unique. The conditions on $x_{{\bf r}}$ given in Lemma \ref{lemma3_complexity} do not determine $\eta[x_{{\bf r}}+1,x_{{\bf r}}+n]$ uniquely.} of length $n$ on $\eta$, and the remainders of $r_j$ modulo $c_j$ for $j\le j_n$ are determined uniquely by the block.
 It follows that
  \begin{equation}\label{formula10_complexity}
  \rho(n)\ge c_1\ldots c_{j_n}.
   \end{equation}
    Remark \ref{rem:averages} yields that
    \begin{equation}\label{formula11_complexity}
    c_1\ldots c_{j_n}\ge j_n^{j_n}.
    \end{equation}
We observed in (\ref{ineq_j_n}) that $j_n\ge \lg_{2\alpha}\left(\frac{\delta n}{2\lg_2n}\right)-1$. Let $0<\varepsilon<1$. The r.h.s of this inequality is greater than $\lg_{2\alpha}(n^{\varepsilon})$ for $n\gg 0$.
Thus, for $n$ big enough, we have:
\begin{equation}
j_n^{j_n}\ge (\varepsilon\lg_{2\alpha}n)^{(\varepsilon\lg_{2\alpha}n)}=\varepsilon^{(\varepsilon\lg_{2\alpha}n)}n^{\varepsilon\lg_{2\alpha}\lg_{2\alpha}n}.
\end{equation}
Putting this together with (\ref{formula10_complexity}) and (\ref{formula11_complexity}) we finish the proof of the proposition.
\end{proof}

\begin{remark}
Analogous (even simpler) arguments can be applied to the example $\cB_1^1$, with the same conclusion about the complexity.
\end{remark}

\section*{Acknowledgments}{Research of the first and second authors is supported by Narodowe Centrum
Nauki grant UMO-2019/33/B/ST1/00364. We are indebted to Daniel Sell for several very helpful remarks and corrections to a previous version of this manuscript.}

\end{document}